\documentclass[11pt]{article}

\usepackage[margin=1in]{geometry}
\usepackage{amsmath,amssymb,amsthm}
\usepackage{array}
\usepackage{booktabs}
\usepackage{graphicx}
\usepackage{xurl}
\usepackage[hidelinks]{hyperref}
\usepackage{cleveref}
\usepackage{microtype}

\theoremstyle{plain}
\newtheorem{theorem}{Theorem}[section]
\newtheorem{lemma}[theorem]{Lemma}
\newtheorem{proposition}[theorem]{Proposition}
\newtheorem{corollary}[theorem]{Corollary}
\theoremstyle{definition}
\newtheorem{definition}[theorem]{Definition}
\newtheorem{problem}[theorem]{Open Problem}
\theoremstyle{remark}

\newcommand{\Rcop}{R_{\mathrm{cop}}}
\newcommand{\Redge}{R_{\mathrm{cop}}^{\mathrm{edge}}}
\newcommand{\Rcl}{R_{\mathrm{cl}}}

\title{Prime Certificates for Exact Vertex-Coprime Ramsey Numbers}
\author{Zhicheng Du\textsuperscript{1} \quad
Wenji Xi\textsuperscript{2} \quad
Zhuo Deng\textsuperscript{1} \quad
Lan Ma\textsuperscript{1}\\[0.5em]
\small \textsuperscript{1}\,Tsinghua Shenzhen International Graduate School,
Tsinghua University\\
\small \textsuperscript{2}\,School of Electrical and Electronic Engineering,
The University of Sheffield}
\date{May 25, 2026}

\begin{document}
\maketitle

\begin{abstract}
Let $G_n$ be the coprime graph on $\{1,\ldots,n\}$.  We prove that the mixed
vertex-coloring coprime Ramsey number satisfies
\[
  \Rcop(k_1,\ldots,k_c)=p_{\sum_{i=1}^c(k_i-1)},
\]
where $p_m$ is the $m$-th prime.  The proof is elementary: the prime clique
$\{1\}\cup\{p\le n:p\text{ prime}\}$ gives the upper bound by pigeonhole,
while a prime-bin partition gives the matching lower bound by coloring each
composite with a bin containing one of its prime divisors.  We reserve
$\Rcop$ for this vertex-coloring parameter; the edge-coloring parameter on
the same host graph is denoted $\Redge$.  The same certificate viewpoint
yields several extensions, including a support-disjointness generalization, a
polynomial-time certificate-extraction primitive, and an exact reduction of
the edge-coloring variant to classical Ramsey numbers:
$\Redge(k_1,\ldots,k_c)=p_{\Rcl(k_1,\ldots,k_c)-1}$.  These two formulas are
rank transfers from the same clique-label certificate.  We also prove that
the balanced two-color diagonal threshold equals the unrestricted threshold
$p_{2k-2}$ for all $k\ge2$, via a deterministic prime-bin split requiring
only the weak inequality $2p_m<p_{2m}<3p_m$; for fixed $c$, a Hall argument
plus a standard Selberg--Delange estimate gives eventual multicolor balanced
certificates.
\end{abstract}



\section{Introduction}

The coprime graph $G_n$ has vertex set $\{1,\ldots,n\}$, with an edge between
two integers exactly when they are coprime.  This paper studies a
vertex-coloring Ramsey problem on this arithmetic host graph: given
$k_1,\ldots,k_c\ge2$, how large must $n$ be before every $c$-coloring of
$[n]$ contains, for some color $i$, a set of $k_i$ pairwise coprime integers
in that color?  We prove the exact formula
\[
  \Rcop(k_1,\ldots,k_c)=p_{\sum_{i=1}^c(k_i-1)}.
\]

The formula is much simpler than one might expect from a direct Ramsey
encoding.  Already for the diagonal case $k=10$, a direct two-color Boolean
satisfiability (SAT) encoding at the threshold has $12{,}474{,}430$ clauses,
while the proof below uses the $19$-vertex prime clique and a matching
prime-bin coloring.
It also rigorously explains prior computational observations of edge-coloring
coprime Ramsey values by identifying the exact classical Ramsey-number
reduction behind them.
Classical edge Ramsey problems on complete graphs remain
asymptotically difficult
\cite{ramsey1930,erdos-szekeres1935,campos2023,campos-r3k2025,
ma-shen-xie2025,hunter-milojevic-sudakov2025,campos-pohoata2026,
attwa-lopez-morris2026,nagda2026,przybocki2026}.  Here the host graph is not
pseudorandom: its large cliques and its extremal colorings are both controlled
by prime divisors.  This arithmetic rigidity is the reason the search
collapses: the obstruction is not a hidden random-like configuration, but the
explicit prime clique
$\{1\}\cup\{p\le n:p\text{ prime}\}$, which gives the upper bound.  The matching
lower bound is the nontrivial direction: below the threshold, every composite
must be colored while keeping each color class below its pairwise-coprime
capacity.  A partition of the available primes into bins does exactly this by
turning prime divisors into injective witnesses.
For instance, the coprime analogue of the classical off-diagonal quantity
$R(3,k)$ is only $p_{k+1}\sim k\log k$, whereas the classical value is
$\Theta(k^2/\log k)$.

The same certificate viewpoint yields several structural extensions and a
boundary theory.  First, the argument extends to support-disjointness graphs
and yields a polynomial-time primitive that recognizes when the certificate
applies.  This is the methodological part of the paper: before invoking a
large search formulation, one can ask whether the host graph is already
explained by disjoint supports over a small set of atoms.  Second, the
edge-coloring variant on the same coprime graph is
not a new arithmetic Ramsey table; it is the classical edge Ramsey table
viewed through prime indices:
\[
  \Redge(k_1,\ldots,k_c)=p_{\Rcl(k_1,\ldots,k_c)-1}.
\]
This notation is deliberate: earlier public computational work used
$R_{\mathrm{cop}}$ for the edge-coloring parameter
\cite{towell-blog,towell-github}; \Cref{thm:edge-reduction} shows that those
values are classical Ramsey numbers viewed through prime indices, not a
separate arithmetic phenomenon.  The main theorem below concerns vertex
colorings of the integers themselves.
The point-edge unification is therefore structural rather than numerical:
the same prime-clique labels control both problems, but vertex colorings
transfer to a pigeonhole threshold while edge colorings transfer to the
classical Ramsey threshold.
\Cref{tab:vertex-edge-comparison} summarizes this distinction before the
detailed related-work table.
This is the prime-index transform principle in its simplest form: once a
Ramsey-type condition is governed by the rank $r=\pi(n)+1$ of the prime-label
clique, the integer threshold is obtained by solving the corresponding rank
problem and then applying $r\mapsto p_{r-1}$.  For vertex colorings the rank
problem is independent bin packing; for edge colorings it is exactly the
classical complete-graph Ramsey problem.
Third, the canonical prime-bin coloring can be highly unbalanced, but the
balanced two-color diagonal endpoint is still exact:
$L_{\mathrm{bal}}(k;2)=p_{2k-2}-1$ for every $k\ge2$.  For every fixed number
of colors, the corresponding multicolor balanced endpoint is also eventually
exact, although small multicolor endpoints can fail.

The result should be read with this scope.  We do not improve classical edge
Ramsey bounds and we do not propose a general SAT-based or artificial
intelligence (AI)-assisted Ramsey solver.  The point is narrower: for this
arithmetic host graph, a large search instance collapses to a prime
certificate.  Boundary
variants such as multicolor balance and shifted intervals then show where
that certificate stops being automatic.
\Cref{fig:certificate-schematic} gives the visual roadmap used throughout
the paper: search first exposes supports, and supports then produce the two
certificates in the proof.

\begin{figure}[htbp]
\centering
\includegraphics[width=0.98\linewidth]{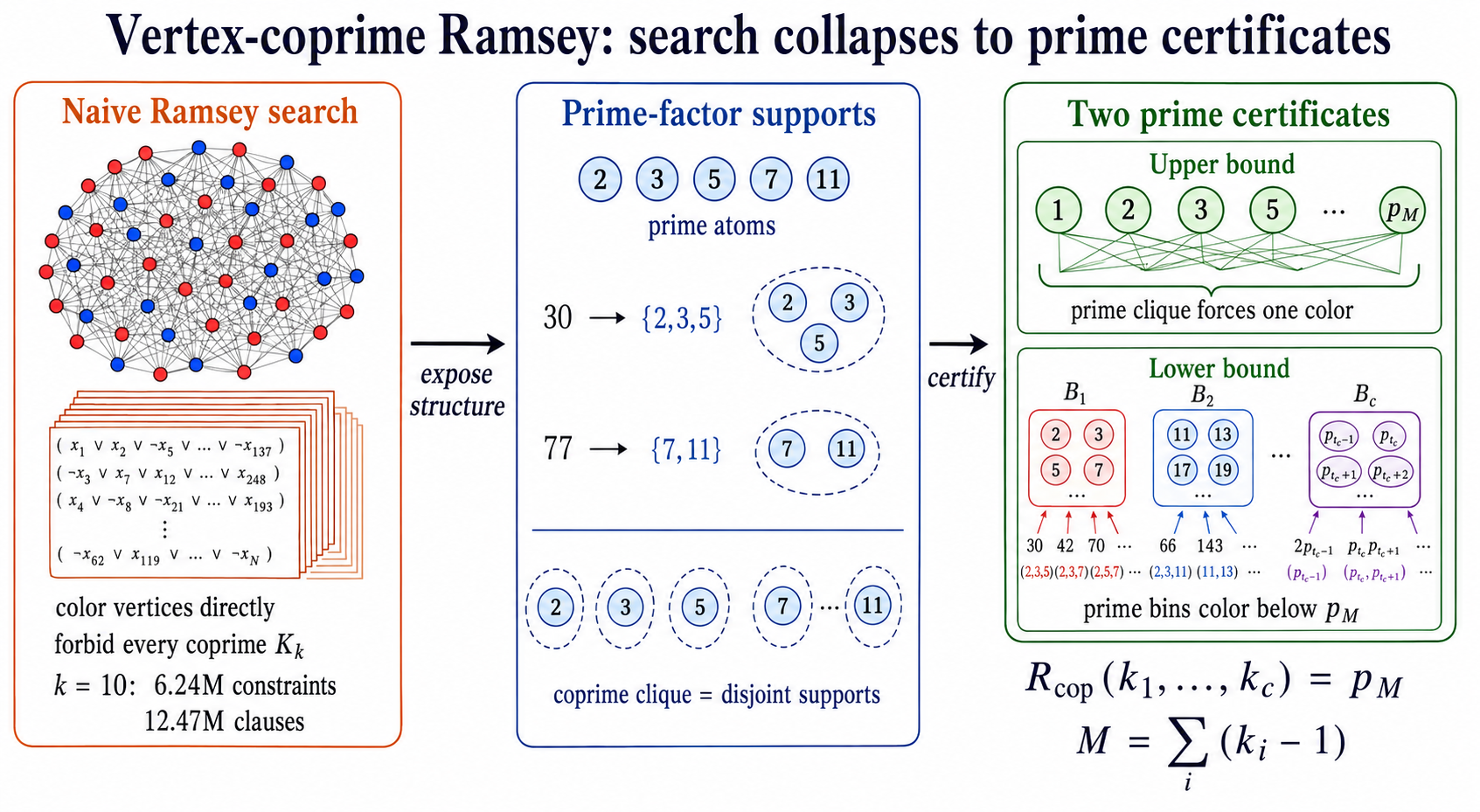}
\caption{The central story of the paper.  A direct Ramsey encoding produces a
large search instance, but prime-factor supports expose the two schematic
certificates that determine the exact threshold: the prime clique for the
upper bound and prime bins for the lower-bound coloring.  The two certificates
are proved in \Cref{thm:main}.}
\label{fig:certificate-schematic}
\end{figure}

\paragraph{Organization.}
After definitions, \Cref{sec:prime-bin} proves the exact mixed formula.
\Cref{sec:extensions} gives the support-disjointness framework, the
certificate primitive, and the clique-label transfer that unifies the vertex
formula with the edge-coloring reduction.
\Cref{sec:computational-context} records finite values and the SAT-scale
diagnostic.  \Cref{sec:balanced} proves the exact balanced two-color endpoint,
and \Cref{sec:boundary} gives the eventual multicolor balanced theorem and
the shifted-interval boundary analysis.
Related work, open problems, and appendices follow.

\section{Definitions}

\begin{definition}[Coprime graph]
For $n\ge 1$, the coprime graph $G_n$ has vertex set $[n]=\{1,\ldots,n\}$.
Distinct vertices $a,b$ are adjacent if and only if $\gcd(a,b)=1$.
\end{definition}

\begin{definition}[Mixed vertex-coprime Ramsey number]
Let $k_1,\ldots,k_c\ge 2$.  The mixed vertex-coprime Ramsey number
$\Rcop(k_1,\ldots,k_c)$ is the least $n$ such that every coloring
\[
  \chi:[n]\to\{1,\ldots,c\}
\]
has some color $i$ whose color class contains a $k_i$-clique of $G_n$.
Equivalently, that color class contains $k_i$ pairwise coprime integers.
\end{definition}

The diagonal notation $\Rcop(k;c)$ means
$\Rcop(\underbrace{k,\ldots,k}_{c\text{ times}})$.  The usual two-color
diagonal notation is $\Rcop(k)=\Rcop(k;2)$; the two-color off-diagonal
notation is $\Rcop(s,t)$.

\section{The Prime-Bin Theorem}
\label{sec:prime-bin}

\begin{lemma}[Clique size in the coprime graph]
\label{lem:max-clique}
The clique number of $G_n$ is $\omega(G_n)=\pi(n)+1$, where $\pi(n)$ counts
the primes at most $n$.
\end{lemma}

\begin{proof}
The set $\{1\}\cup\{p\le n:p\text{ prime}\}$ is a clique, so
$\omega(G_n)\ge \pi(n)+1$.
Conversely, in any clique of $G_n$, each vertex other than $1$ has at least
one prime divisor, and two distinct non-one vertices in the clique cannot
share a prime divisor.  Choosing one prime divisor from each non-one vertex
therefore injects the non-one vertices into the set of primes at most $n$.
Thus a clique has at most $\pi(n)$ vertices different from $1$, plus possibly
the vertex $1$.
\end{proof}

\begin{theorem}[Exact mixed formula]
\label{thm:main}
For all $c\ge 1$ and $k_1,\ldots,k_c\ge 2$,
\[
  \Rcop(k_1,\ldots,k_c)=p_M,
  \qquad
  M=\sum_{i=1}^c(k_i-1).
\]
\end{theorem}

\begin{proof}
First let $n=p_M$.  The prime clique
\[
  P_n=\{1\}\cup\{p\le n:p\text{ prime}\}
\]
has $M+1$ vertices.  In any $c$-coloring, if color $i$ had at most
$k_i-1$ vertices of $P_n$ for every $i$, then $P_n$ would contain at most
$\sum_i(k_i-1)=M$ vertices, a contradiction.  Hence some color $i$ contains
at least $k_i$ vertices of $P_n$, and these form a monochromatic
$k_i$-clique.  Therefore $\Rcop(k_1,\ldots,k_c)\le p_M$.

Now let $n<p_M$.  Then $\pi(n)\le M-1$.  Partition the primes at most $n$
into bins $B_1,\ldots,B_c$ with capacities
\[
  |B_1|\le k_1-2,\qquad |B_i|\le k_i-1\quad (2\le i\le c).
\]
This is possible because the total capacity is
\[
  (k_1-2)+\sum_{i=2}^c(k_i-1)=M-1.
\]
Color vertex $1$ with color $1$.  For every integer $m>1$, choose one prime
divisor $q(m)$ of $m$ and color $m$ by a bin containing $q(m)$.

Consider a monochromatic clique in color $i\ge 2$.  The chosen witness
primes $q(m)$ for its vertices are all distinct, since two pairwise coprime
integers cannot share a prime divisor.  Hence the clique has size at most
$|B_i|\le k_i-1$.  For color $1$, the same argument bounds the non-one
vertices by $|B_1|\le k_1-2$, and vertex $1$ can add at most one more
vertex, so the clique size is at most $k_1-1$.  Thus this coloring avoids
all forbidden monochromatic cliques, proving
$\Rcop(k_1,\ldots,k_c)>n$ for every $n<p_M$.
\end{proof}

\begin{corollary}[Diagonal and off-diagonal forms]
\label{cor:forms}
For $k,c\ge 2$,
\[
  \Rcop(k;c)=p_{c(k-1)}.
\]
For two colors with asymmetric demands $s,t\ge 2$,
\[
  \Rcop(s,t)=p_{s+t-2}.
\]
\end{corollary}

\begin{corollary}[Covering version]
\label{cor:covering}
Let $C_{\mathrm{cop}}(k_1,\ldots,k_c)$ be the least $n$ such that there is no
cover
\[
  [n]=A_1\cup\cdots\cup A_c
\]
with each $A_i$ containing no $k_i$ pairwise coprime integers.  Then
\[
  C_{\mathrm{cop}}(k_1,\ldots,k_c)=p_{\sum_i(k_i-1)}.
\]
\end{corollary}

\begin{proof}
Let $M=\sum_i(k_i-1)$.  At $n=p_M$, the prime clique has $M+1$ vertices.  If
it were covered by sets $A_i$ with no $k_i$ pairwise coprime vertices, then
each $A_i$ would contain at most $k_i-1$ vertices of this clique, so the total
covering capacity would be at most $M$, impossible.  For $n<p_M$, the
prime-bin coloring in the proof of \Cref{thm:main} is in particular a cover by
the color classes, so the same lower-bound construction applies.
\end{proof}

\begin{corollary}[Asymptotics]
For fixed $c$,
\[
  \Rcop(k;c)\sim c(k-1)\log(c(k-1)).
\]
More precisely, with $m=c(k-1)$,
\[
  \Rcop(k;c)
  =m\bigl(\log m+\log\log m-1+o(1)\bigr).
\]
In particular, $\Rcop(k;2)=\Theta(k\log k)$.
\end{corollary}

\begin{proof}
This is the standard asymptotic expansion for the $m$-th prime, applied with
$m=c(k-1)$.
\end{proof}

\begin{corollary}[Quantitative gap from classical off-diagonal Ramsey]
\label{cor:r3k-gap}
For the classical complete-graph edge Ramsey number $\Rcl(3,k)$,
\[
  \frac{\Rcop(3,k)}{\Rcl(3,k)}
  =\Theta\!\left(\frac{(\log k)^2}{k}\right)\to0 .
\]
\end{corollary}

\begin{proof}
By \Cref{cor:forms}, $\Rcop(3,k)=p_{k+1}\sim k\log k$.  The classical estimate
$\Rcl(3,k)=\Theta(k^2/\log k)$ follows from the upper bound of
Ajtai--Koml\'os--Szemer\'edi and the lower bound of Kim
\cite{ajtai-komlos-szemeredi1980,kim1995}.  Dividing the two estimates gives
the claim.
\end{proof}

\section{Three Certificate Extensions}
\label{sec:extensions}

The prime-bin proof is not only a proof for one vertex-coloring problem.  It
identifies exactly which structure is being used: every non-universal vertex
carries a nonempty set of prime atoms, and pairwise-coprime cliques inject
into disjoint atoms.  Three immediate consequences follow: a
support-disjointness generalization, an algorithmic primitive that
recognizes that support model on an arbitrary host graph, and an exact
reduction of the edge-coloring variant to classical Ramsey numbers.

\subsection{Support-Disjointness Graphs}

Let $A$ be a finite set of atoms with $|A|=r$.  A support-disjointness graph
is a graph whose vertices carry supports $\sigma(v)\subseteq A$, with
vertices adjacent if and only if their supports are disjoint.  Assume that every
singleton support $\{a\}$, $a\in A$, occurs as a vertex.  There are two cases
we need:

\begin{itemize}
  \item the \emph{one-universal} case, where exactly one vertex has empty
  support;
  \item the \emph{no-universal} case, where every vertex has nonempty
  support.
\end{itemize}
Vertices with full support $A$ are allowed in the no-universal case.  They
are isolated, since their support is disjoint from no nonempty support, and
therefore they never help form a clique of size at least two.

\begin{theorem}[Support-disjointness Ramsey theorem]
\label{thm:support}
For mixed vertex demands $k_1,\ldots,k_c\ge 2$, let
$M=\sum_i(k_i-1)$.  In the one-universal case, every $c$-coloring forces a
monochromatic $K_{k_i}$ in some color $i$ if and only if $r\ge M$.  In the
no-universal case, the corresponding condition is $r\ge M+1$.
\end{theorem}

\begin{proof}
We prove the two directions separately.  In the one-universal case, if
$r\ge M$, the singleton-support vertices together with the universal vertex
form a clique of size $r+1\ge M+1$.  If no color $i$ contained $k_i$ vertices
on this clique, the clique would have at most $\sum_i(k_i-1)=M$ vertices, a
contradiction.  Thus every coloring forces a forbidden monochromatic clique.
In the no-universal case, the singleton-support vertices form a clique of
size $r$, so the same pigeonhole argument forces a clique exactly when
$r\ge M+1$.

Conversely, suppose first that we are in the one-universal case and
$r<M$, so $r\le M-1$.  Color the empty-support vertex with color $1$ and
partition the atoms into bins $B_1,\ldots,B_c$ with
$|B_1|\le k_1-2$ and $|B_i|\le k_i-1$ for $i\ge2$; these capacities sum to
$M-1$.  Since the bins cover $A$, every nonempty support meets at least one
bin.  Color each nonempty-support vertex by any bin meeting its support.
For a monochromatic clique in color $i$, choose one atom in
$\sigma(v)\cap B_i$ from each nonempty-support vertex $v$ in the clique.
The supports in a clique are pairwise disjoint, so these chosen atoms are
distinct.  Hence color $1$ has clique size at most $|B_1|+1\le k_1-1$, and
each other color $i$ has clique size at most $|B_i|\le k_i-1$.

In the no-universal case, when $r<M+1$ we have $r\le M$.  Partition the
atoms into bins of capacities $k_1-1,\ldots,k_c-1$ and repeat the same
support-injection coloring.  There is no empty-support vertex, so every
monochromatic clique in color $i$ has size at most $k_i-1$.
\end{proof}

\Cref{thm:main} is the one-universal case with atoms equal to primes
$\le n$ and the empty-support vertex equal to $1$.  The proper-divisor
coprime graph $\Gamma_N$ of one composite integer $N$, whose vertices are
proper divisors $d$ with adjacency $\gcd(d,e)=1$, is the no-universal case
with atoms the prime divisors of $N$, after isolated full-support divisors
are allowed but never help form cliques.  Thus its mixed vertex Ramsey
threshold is controlled by $\nu(N)$, the number of prime divisors of $N$.
The following corollary records the squarefree-kernel case explicitly.

\begin{corollary}[Squarefree-kernel coprime graph]
\label{cor:squarefree-kernel}
Let $G_n^{\mathrm{rad}}$ have vertex set $[n]$ and join $a,b$ when
\[
  \gcd(\operatorname{rad}(a),\operatorname{rad}(b))=1,
\]
where $\operatorname{rad}(a)$ is the product of the distinct primes dividing
$a$.  The mixed vertex Ramsey number on $G_n^{\mathrm{rad}}$ is
\[
  R_{\mathrm{cop}}^{\mathrm{rad}}(k_1,\ldots,k_c)
  =p_{\sum_i(k_i-1)}.
\]
\end{corollary}

\begin{proof}
The support of $a$ is exactly the set of primes dividing
$\operatorname{rad}(a)$, and disjointness of these supports is equivalent to
$\gcd(\operatorname{rad}(a),\operatorname{rad}(b))=1$.  Thus
$G_n^{\mathrm{rad}}$ is the same one-universal support-disjointness graph as
$G_n$, with the same singleton supports and the same empty-support vertex
$1$.  Apply \Cref{thm:support}.
\end{proof}

The last row of \Cref{tab:support-examples} is equally important: if a host
graph adds edges not explained by disjoint supports, then the injection into
atom bins can fail and the theorem is no longer available for free.  This is
a structural restriction, not a technicality: the support theorem applies to
pure support-disjointness hosts and to variants that preserve that adjacency
rule.

\begin{table}[htbp]
\centering
\caption{The prime-bin theorem is a support-disjointness statement, not an
accident of the interval $[n]$.  The last row warns that extra edges destroy
the certificate.}
\label{tab:support-examples}
\small
\begin{tabular}{@{}>{\raggedright\arraybackslash}p{0.30\linewidth}>{\raggedright\arraybackslash}p{0.28\linewidth}>{\raggedright\arraybackslash}p{0.32\linewidth}@{}}
\toprule
Host graph & Resource rank & Ramsey threshold in resource rank \\
\midrule
Integer coprime graph $G_n$ on $[n]$ & $\pi(n)$ prime atoms plus vertex $1$ & force iff $\pi(n)\ge \sum_i(k_i-1)$ \\
Support-disjointness graph with one empty-support vertex & $r$ atoms plus one universal vertex & force iff $r\ge \sum_i(k_i-1)$ \\
Proper-divisor coprime graph $\Gamma_N$ & $\nu(N)$ prime divisors, no vertex $1$ & force iff $\nu(N)\ge 1+\sum_i(k_i-1)$ \\
Squarefree-kernel version on $[n]$ & same prime supports as $G_n$ & identical threshold to $G_n$ \\
Graphs with extra edges beyond support-disjointness, e.g. nontrivial $\gcd(a,b)\in D$ rules & supports no longer control all adjacencies & not covered by the theorem \\
\bottomrule
\end{tabular}

\end{table}

\subsection{A Certificate-Extraction Primitive}
\label{subsec:certificate-primitive}

The support theorem can be used as a small algorithmic primitive.  It is not
a general Ramsey solver; rather, it recognizes when a Ramsey instance has the
same support-disjointness certificate as the coprime graph.
\Cref{prop:certificate-primitive} states the primitive in the form used by
the reproducibility scripts.

\begin{proposition}[Support certificate primitive]
\label{prop:certificate-primitive}
Suppose a finite graph $H=(V,E)$ is given together with bitset supports
$\sigma(v)\subseteq A$, $|A|=r$, and mixed vertex demands
$k_1,\ldots,k_c$.  In $O(|V|^2r)$ bit operations, one can verify the following
certificate conditions:
\begin{enumerate}
  \item every singleton support occurs;
  \item the number of empty-support vertices is either zero or one;
  \item for every pair $u,v$, $uv\in E$ if and only if
  $\sigma(u)\cap\sigma(v)=\emptyset$.
\end{enumerate}
If the support conditions hold, the algorithm returns the exact forcing
condition in \Cref{thm:support}.  If the rank is in the forcing range it
outputs the singleton-support clique, with the empty-support vertex if
present.  If the rank is below the forcing range it outputs the atom-bin
avoiding coloring from the proof of \Cref{thm:support}.
\end{proposition}

\begin{proof}
Represent supports as bitsets.  Singleton coverage and the number of empty
supports are checked by one scan over $V$.  The edge condition is checked by
testing, for each unordered pair $u,v$, whether the bitset intersection
$\sigma(u)\cap\sigma(v)$ is empty and comparing this with adjacency in
$H$, which costs $O(|V|^2r)$ bit operations in the naive representation.
After these checks pass, the output is exactly the constructive proof of
\Cref{thm:support}: the forcing certificate is the clique of singleton
supports (plus the empty-support vertex in the one-universal case), and the
avoiding certificate is obtained by partitioning atoms into the bin
capacities used in the lower-bound proof.
\end{proof}

This primitive is the algorithmic form of the paper's main lesson.  A direct
SAT encoding asks the solver to discover millions of anti-clique clauses; a
support certificate first asks whether all adjacency is already explained by
disjoint resource atoms.  When the answer is yes, the Ramsey threshold and
both certificates are obtained without search.  When the answer is no, as in
host graphs with additional $\gcd(a,b)\in D$ edges, the primitive fails
explicitly rather than silently suggesting that the prime-bin proof still
applies.  \Cref{tab:support-primitive} records the corresponding small
recognition outputs.

\begin{table}[htbp]
\centering
\caption{Small outputs of the support-certificate primitive.  Shifted
intervals still have support-disjoint adjacency, but they may fail singleton
coverage for the atom set induced by the interval.}
\label{tab:support-primitive}
\small
\begin{tabular}{@{}>{\raggedright\arraybackslash}p{0.22\linewidth}>{\raggedright\arraybackslash}p{0.20\linewidth}>{\raggedright\arraybackslash}p{0.46\linewidth}@{}}
\toprule
Instance & Primitive result & Explanation \\
\midrule
$G_{30}$ & pass & passes: one universal vertex and all prime singletons occur \\
$[11,17]$ & fail & fails singleton coverage for atoms $3,5,7$ \\
$G_{30}$ plus edge $6$--$10$ & fail & fails adjacency iff support-disjointness \\
\bottomrule
\end{tabular}

\end{table}

\subsection{Edge-Coloring Reduction}

Let $\Rcl(k_1,\ldots,k_c)$ denote the classical multicolor edge Ramsey number
for complete graphs: the least $N$ such that every $c$-edge-coloring of
$K_N$ contains a monochromatic $K_{k_i}$ in some color $i$.  Define
$\Redge(k_1,\ldots,k_c)$ similarly, but with edges of $G_n$ colored instead
of edges of $K_n$.

\begin{theorem}[Edge-coprime Ramsey reduction]
\label{thm:edge-reduction}
For all $k_1,\ldots,k_c\ge 2$,
\[
  \Redge(k_1,\ldots,k_c)=p_{\Rcl(k_1,\ldots,k_c)-1}.
\]
\end{theorem}

\begin{proof}
Let $R=\Rcl(k_1,\ldots,k_c)$.  At $n=p_{R-1}$, the prime clique
$\{1\}\cup\{p\le n:p\text{ prime}\}$ has $R$ vertices.  Any edge-coloring of
$G_n$ restricted to this clique is an edge-coloring of $K_R$, so it contains
a monochromatic $K_{k_i}$ for some $i$.  Hence $\Redge\le p_{R-1}$.

For $n<p_{R-1}$, the label set
\[
  L_n=\{*\}\cup\{p\le n:p\text{ prime}\}
\]
has at most $R-1$ elements.  By the definition of $R$, there is a classical
$c$-edge-coloring of $K_{R-1}$ with no forbidden monochromatic clique; if
$|L_n|<R-1$, restrict such a coloring to any $|L_n|$ labels.  Label integer
$1$ by $*$ and every $m>1$ by one chosen prime divisor $\ell(m)$.
Color a coprime edge $\{a,b\}$ of $G_n$ by the color of
$\{\ell(a),\ell(b)\}$ in the classical coloring.  In any coprime clique of
$G_n$, these labels are distinct: two non-one coprime integers cannot share a
chosen prime divisor, and only vertex $1$ has label $*$.  Therefore a
monochromatic forbidden clique in $G_n$ would map to one in the classical
coloring, a contradiction.  Thus $\Redge>n$ for every $n<p_{R-1}$.
\end{proof}

\Cref{tab:edge-coprime} lists the first edge-coprime values and bounds
obtained from \Cref{thm:edge-reduction}; these numbers should be read as
classical Ramsey data viewed through the prime-index map, not as new
arithmetic constants.

\begin{table}[htbp]
\centering
\caption{Edge-coprime Ramsey values are prime-index images of classical
complete-graph Ramsey values.}
\label{tab:edge-coprime}
\small
\begin{tabular}{@{}lrrr@{}}
\toprule
Classical value & classical $R$ & edge-coprime value & status \\
\midrule
R(3,3) & 6 & 11 & exact \\
R(4,4) & 18 & 59 & exact \\
R(5,5) & 43--46 & 181--197 & best-known bounds \\
R(3,4) & 9 & 19 & exact \\
R(3,5) & 14 & 41 & exact \\
\bottomrule
\end{tabular}

\end{table}

\subsection{A Common Clique-Label Mechanism}
\label{subsec:clique-label}

The vertex theorem and the edge reduction are two projections of the same
certificate.  Let $H$ be a finite graph.  A \emph{clique-label certificate of
rank $r$} consists of a clique $C\subseteq V(H)$ with $|C|=r$ and a map
\[
  \lambda:V(H)\to C
\]
such that $\lambda$ is the identity on $C$ and is injective on every clique of
$H$.  Equivalently, it is enough to check that adjacent vertices receive
distinct labels: any two distinct vertices in a clique are adjacent, so the
labels on a clique are then pairwise distinct.

\begin{proposition}[Rank Ramsey transfer]
\label{prop:rank-transfer}
Suppose $H$ has a clique-label certificate of rank $r$.  For vertex colorings,
every $c$-coloring of $V(H)$ contains a monochromatic $K_{k_i}$ in some color
$i$ if and only if
\[
  r\ge 1+\sum_{i=1}^c(k_i-1).
\]
For edge colorings, every $c$-edge-coloring of $H$ contains such a
monochromatic clique if and only if
\[
  r\ge \Rcl(k_1,\ldots,k_c).
\]
\end{proposition}

\begin{proof}
For vertex colorings, the upper bound is the pigeonhole principle on the
clique $C$.  If $r\le \sum_i(k_i-1)$, color the labels in $C$ so that color
$i$ receives at most $k_i-1$ labels, and color each vertex $v$ by the color of
$\lambda(v)$.  A monochromatic clique in $H$ injects into labels of the same
color, so it has size at most $k_i-1$.

For edge colorings, the upper bound is the classical Ramsey theorem applied to
the clique $C$.  If $r<\Rcl(k_1,\ldots,k_c)$, choose a classical
$c$-edge-coloring of $K_r$ with no forbidden monochromatic clique, and color
each edge $\{u,v\}$ of $H$ by the color of the label edge
$\{\lambda(u),\lambda(v)\}$.  The labels are distinct on every edge, and a
monochromatic clique of $H$ would inject into a monochromatic clique of
$K_r$, a contradiction.
\end{proof}

\Cref{prop:rank-transfer} is the formal point-edge unification used below.
The next proposition records how far the same transfer extends when the edge
target is not necessarily a clique.

\begin{proposition}[Substructure transfer and its boundary]
\label{prop:substructure-transfer}
Let $\mathcal F_1,\ldots,\mathcal F_c$ be finite graph families, and let
$R_{\mathrm{cl}}(\mathcal F_1,\ldots,\mathcal F_c)$ be the least $R$ such
that every $c$-edge-coloring of $K_R$ contains, for some color $i$, a
color-$i$ copy of a graph in $\mathcal F_i$.  If $H$ has a clique-label
certificate of rank $r$ and
\[
  r\ge R_{\mathrm{cl}}(\mathcal F_1,\ldots,\mathcal F_c),
\]
then every $c$-edge-coloring of $H$ contains such a monochromatic
substructure.  If every graph in every $\mathcal F_i$ is complete, this
condition is also necessary.
\end{proposition}

\begin{proof}
The upper bound again restricts the coloring to the rank-$r$ clique $C$.
For necessity in the complete-target case, suppose
$r<R_{\mathrm{cl}}(\mathcal F_1,\ldots,\mathcal F_c)$ and choose an avoiding
coloring of $K_r$.  Pull it back to $H$ by coloring $uv$ according to
$\lambda(u)\lambda(v)$.  Any forbidden complete target in $H$ is a clique,
so its labels are distinct and form the same forbidden target in $K_r$,
a contradiction.
\end{proof}

The completeness assumption is the exact boundary of the transfer.  For
paths, cycles, trees, and other non-complete targets, the label map need not
be injective on non-adjacent vertices of a witness.  Thus the prime clique
still gives a classical-Ramsey upper bound, but the pullback lower bound can
be strict.  The same proof also applies to fixed color-pattern conditions on
complete witnesses, such as the Gallai condition ``monochromatic or rainbow
triangle'', because the witness vertices again form a clique.

For the coprime graph $G_n$, take
\[
  C_n=\{1\}\cup\{p\le n:p\text{ prime}\},\qquad r(n)=\pi(n)+1,
\]
and let $\lambda(1)=1$ while $\lambda(m)$ is any chosen prime divisor of
$m>1$.  Pairwise coprime integers cannot share a chosen prime divisor, so
$\lambda$ is injective on every coprime clique.  Thus the vertex threshold is
the rank condition $r(n)\ge 1+\sum_i(k_i-1)$, while the edge threshold is the
rank condition $r(n)\ge\Rcl(k_1,\ldots,k_c)$.  The resulting integer thresholds
are exactly the prime-index formulas of \Cref{thm:main,thm:edge-reduction}.
\Cref{tab:unified-rank-threshold} displays the two rank triggers side by side.

\begin{table}[htbp]
\centering
\caption{The same clique-label certificate gives different rank triggers:
pigeonhole for vertex colorings and classical Ramsey for edge colorings.}
\label{tab:unified-rank-threshold}
\small
\begin{tabular}{@{}>{\raggedright\arraybackslash}p{0.22\linewidth}>{\raggedright\arraybackslash}p{0.26\linewidth}rr@{}}
\toprule
Formulation & base threshold on $K_r$ & rank trigger & coprime threshold \\
\midrule
$\Rcop(3;2)$ & $1+(2 + 2)=5$ & 5 & $p_{4}=7$ \\
$\Rcop(4;2)$ & $1+(3 + 3)=7$ & 7 & $p_{6}=13$ \\
$\Rcop(3;3)$ & $1+(2 + 2 + 2)=7$ & 7 & $p_{6}=13$ \\
$\Rcop(3,5)$ & $1+(2 + 4)=7$ & 7 & $p_{6}=13$ \\
$\Redge(3;2)$ & $R=6$ & 6 & $p_{5}=11$ \\
$\Redge(3;3)$ & $R=17$ & 17 & $p_{16}=53$ \\
$\Redge(4;2)$ & $R=18$ & 18 & $p_{17}=59$ \\
$\Redge(3,4)$ & $R=9$ & 9 & $p_{8}=19$ \\
$\Redge(3,5)$ & $R=14$ & 14 & $p_{13}=41$ \\
\bottomrule
\end{tabular}

\end{table}

\begin{corollary}[Scaled gcd-$d$ edge-clique variant]
\label{cor:gcd-scaling}
Let $G_{n,d}$ have vertex set $[n]$ and edges $ab$ exactly when
$\gcd(a,b)=d$.  The edge-coloring clique Ramsey threshold on $G_{n,d}$ is
\[
  R_{\gcd=d}^{\mathrm{edge}}(k_1,\ldots,k_c)
  =d\,p_{\Rcl(k_1,\ldots,k_c)-1}.
\]
\end{corollary}

\begin{proof}
Vertices in a nontrivial edge of $G_{n,d}$ are multiples of $d$, and the map
$dm\mapsto m$ identifies the non-isolated part of $G_{n,d}$ with
$G_{\lfloor n/d\rfloor}$.  The least $n$ for which this scaled copy reaches
the edge-coprime threshold of \Cref{thm:edge-reduction} is therefore
$d\,p_{\Rcl(k_1,\ldots,k_c)-1}$.
\end{proof}

The scaled gcd variant in \Cref{cor:gcd-scaling} is a simple example where
the same transfer survives after an arithmetic rescaling of the host graph.

\begin{corollary}[Uniform hypergraph vertex-coloring analogue]
\label{cor:hypergraph-vertex}
Fix $t\ge2$.  Let $G_n^{(t)}$ be the $t$-uniform hypergraph whose hyperedges
are the $t$-element pairwise-coprime subsets of $[n]$.  For $k_i\ge t$, define
the vertex-coloring threshold by asking for a monochromatic complete
$t$-uniform hypergraph on $k_i$ vertices in some color $i$.  Then
\[
  R_{\mathrm{cop}}^{(t)\text{-vertex}}(k_1,\ldots,k_c)
  =p_{\sum_i(k_i-1)}.
\]
\end{corollary}

\begin{proof}
For $k_i\ge t$, a $k_i$-vertex complete subhypergraph of $G_n^{(t)}$ is
equivalent to a $k_i$-set of pairwise coprime integers: pairwise coprimality
clearly gives every $t$-edge, and any two vertices in a complete
$t$-uniform witness lie together in some hyperedge.  Thus the monochromatic
witnesses are exactly the same pairwise-coprime vertex sets as in
\Cref{thm:main}, so the same prime clique and prime-bin coloring give the
same threshold.
\end{proof}

\begin{corollary}[Uniform hypergraph edge-coloring analogue]
\label{cor:hypergraph-edge}
Fix $t\ge2$.  Let $G_n^{(t)}$ be the $t$-uniform hypergraph whose hyperedges
are the $t$-element pairwise-coprime subsets of $[n]$.  For $k_i\ge t$, if
$R_{\mathrm{cl}}^{(t)}(k_1,\ldots,k_c)$ denotes the classical $t$-uniform
complete-hypergraph Ramsey number, then the hyperedge-coloring threshold is
\[
  R_{\mathrm{cop}}^{(t)\text{-edge}}(k_1,\ldots,k_c)
  =p_{R_{\mathrm{cl}}^{(t)}(k_1,\ldots,k_c)-1}.
\]
\end{corollary}

\begin{proof}
The prime-label clique is a complete $t$-uniform hypergraph of rank
$\pi(n)+1$.  Below the stated threshold, pull back a classical avoiding
coloring of the complete $t$-uniform hypergraph on the label set.  The labels
are injective on every pairwise-coprime vertex set, so a monochromatic
complete $t$-uniform hypergraph in $G_n^{(t)}$ would give one on the labels.
\end{proof}

\Cref{cor:hypergraph-vertex,cor:hypergraph-edge} show that the same
certificate language also covers complete uniform-hypergraph witnesses.
\Cref{tab:coprime-ramsey-hierarchy} summarizes which variants remain inside
the certificate regime and which ones are useful boundary tests.

\begin{table}[htbp]
\centering
\caption{Coprime Ramsey hierarchy exposed by the prime-label certificate.
The lower levels have closed formulas; higher levels are useful boundary
tests because they deliberately remove or constrain part of the certificate.}
\label{tab:coprime-ramsey-hierarchy}
\small
\begin{tabular}{@{}>{\raggedright\arraybackslash}p{0.13\linewidth}>{\raggedright\arraybackslash}p{0.28\linewidth}>{\raggedright\arraybackslash}p{0.46\linewidth}@{}}
\toprule
Level & Governing mechanism & Status in this paper \\
\midrule
0 & Prime-clique rank & $\omega(G_n)=\pi(n)+1$ and $\chi(G_n)=\pi(n)+1$ give the $k=2$ base case. \\
1 & Prime-bin packing & Mixed vertex colorings, including the pairwise-coprime uniform hypergraph vertex analogue, collapse to $\Rcop(k_1,\ldots,k_c)=p_{\sum_i(k_i-1)}$. \\
2 & Complete-witness Ramsey pullback & Clique, Gallai-clique, scaled gcd-$d$, and hypergraph-clique edge targets are prime-index transfers. \\
3 & Density-constrained certificates & Two-color balance is exact; multicolor balance has finite defects and an eventual certificate regime. \\
4 & Local arithmetic dependence & Shifted intervals lose vertex $1$ and depend on local prime-factor structure. \\
\bottomrule
\end{tabular}

\end{table}

\section{Numerical Values and Computational Context}
\label{sec:computational-context}

The theorem gives all values, so finite SAT search is not needed for the
vertex-coloring problem.  \Cref{tab:values} gives the beginning of the
two-color diagonal sequence, and \Cref{fig:growth} compares it with the
first-order scale $2k\log k$.

\begin{table}[htbp]
\centering
\caption{Exact two-color diagonal vertex-coprime Ramsey numbers.}
\label{tab:values}
\begin{tabular}{rrrr}
\toprule
$k$ & prime index & $R_{\mathrm{cop}}(k;2)$ & $R/(k\log k)$ \\
\midrule
2 & 2 & 3 & 2.164 \\
3 & 4 & 7 & 2.124 \\
4 & 6 & 13 & 2.344 \\
5 & 8 & 19 & 2.361 \\
6 & 10 & 29 & 2.698 \\
7 & 12 & 37 & 2.716 \\
8 & 14 & 43 & 2.585 \\
9 & 16 & 53 & 2.680 \\
10 & 18 & 61 & 2.649 \\
11 & 20 & 71 & 2.692 \\
12 & 22 & 79 & 2.649 \\
13 & 24 & 89 & 2.669 \\
14 & 26 & 101 & 2.734 \\
15 & 28 & 107 & 2.634 \\
16 & 30 & 113 & 2.547 \\
17 & 32 & 131 & 2.720 \\
18 & 34 & 139 & 2.672 \\
19 & 36 & 151 & 2.699 \\
20 & 38 & 163 & 2.721 \\
21 & 40 & 173 & 2.706 \\
\bottomrule
\end{tabular}

\end{table}

\begin{figure}[htbp]
\centering
\includegraphics[width=0.76\linewidth]{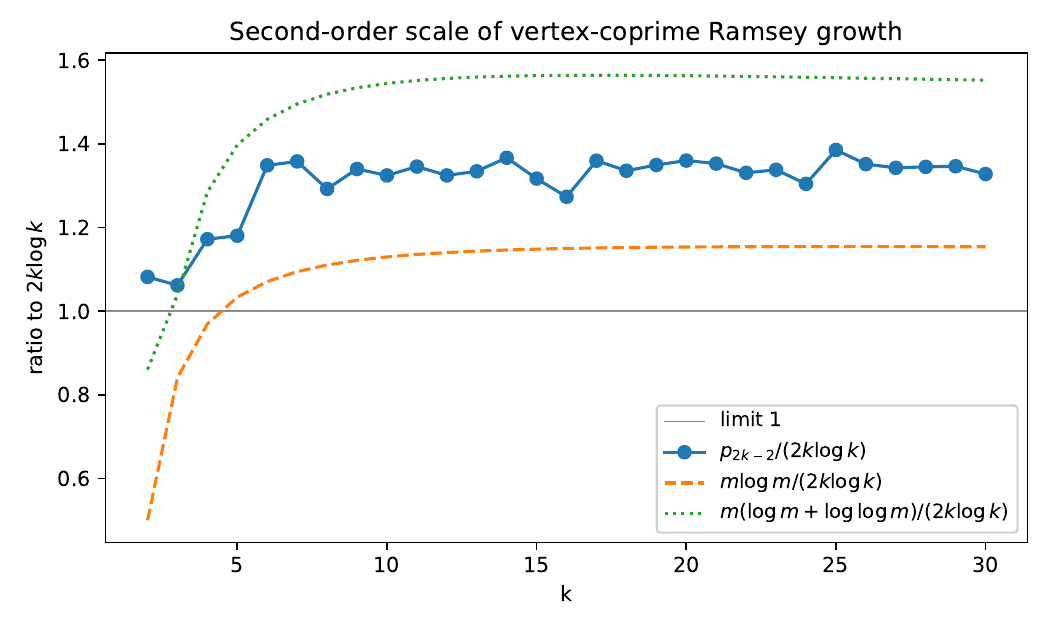}
\caption{Ratio plot for $\Rcop(k;2)=p_{2k-2}$ against the baseline
$2k\log k$.  The second-order prime-number-theorem term explains the visible
gap beyond the first-order scale.}
\label{fig:growth}
\end{figure}

The value $\Rcop(10;2)$ is $p_{18}=61$.  The extremal coloring at
$n=60$ is obtained by placing $1$ and the first eight primes in one bin,
placing the remaining nine primes $\le 60$ in the other bin, and coloring
each composite by a bin containing one of its prime divisors.  This shows
directly that $n=60$ is still avoidable, while the prime clique at $n=61$
has $19$ vertices and forces ten vertices in one of two colors.

This corrects the superseded exploratory value $53$ for $\Rcop(10;2)$; the SAT
encoding issue behind that false lead is discussed in \Cref{app:sat-forensics}.

\subsection{From Large SAT Formulas to a Small Certificate}

A useful way to understand the result is to compare the direct computational
encoding with the proof certificate.  A direct two-color SAT encoding has one
Boolean variable per vertex-color pair and two anti-monochromatic clauses for
every coprime $k$-clique.  At the exact threshold, the formula quickly becomes
large even though the proof of unsatisfiability uses only the prime clique.

\begin{table}[htbp]
\centering
\caption{Direct SAT scale at the exact threshold compared with the prime-clique certificate.}
\label{tab:structural-diagnostics}
\small
\begin{tabular}{rrrrrr}
\toprule
$k$ & $R_{\mathrm{cop}}(k;2)$ & cert. rank & all $K_k$ & prime-clique $K_k$ & SAT clauses \\
\midrule
3 & 7 & 5 & 19 & 10 & 52 \\
4 & 13 & 7 & 151 & 35 & 328 \\
5 & 19 & 9 & 831 & 126 & 1,700 \\
6 & 29 & 11 & 7,803 & 462 & 15,664 \\
7 & 37 & 13 & 42,708 & 1,716 & 85,490 \\
8 & 43 & 15 & 186,945 & 6,435 & 373,976 \\
9 & 53 & 17 & 1,280,587 & 24,310 & 2,561,280 \\
10 & 61 & 19 & 6,237,154 & 92,378 & 12,474,430 \\
\bottomrule
\end{tabular}

\end{table}

\Cref{tab:structural-diagnostics} is not needed for the theorem, but it
explains why the computational path is misleading.  From $k=9$ to $k=10$,
the direct encoding grows from $1{,}280{,}587$ to $6{,}237{,}154$ coprime
$K_k$ constraints, while the certificate rank grows only from $17$ to $19$.
For $k=10$, the direct encoding has over $12$ million clauses, whereas the
mathematical proof only asks how a two-coloring splits the $19$ vertices of
the prime clique at $n=61$.  The apparent combinatorial explosion is real for
brute-force encodings, but irrelevant once the prime-factor certificate is
exposed.

The clique-label viewpoint in \Cref{subsec:clique-label} is also easy to
check directly.  Using the least prime divisor as the label, we checked on
finite instances that every coprime edge receives two distinct labels for the values in
\Cref{tab:clique-label-check}.  Since every pair of vertices in a clique is
adjacent, this no-collision edge property is exactly the local condition that
makes the label map injective on cliques.  The finite table is not a proof
ingredient; it is a reproducibility diagnostic for the mechanism behind both
the vertex and edge reductions.

\begin{table}[htbp]
\centering
\caption{Finite check of the no-label-collision property for $G_n$.  A collision
would mean a coprime edge whose endpoints receive the same prime label; none
occur in the checked range.  For cliques, checking all adjacent pairs is the
same local condition as label injectivity.}
\label{tab:clique-label-check}
\small
\begin{tabular}{@{}rrrrr@{}}
\toprule
$n$ & rank $r=\pi(n)+1$ & coprime edges checked & edge label collisions & result \\
\midrule
10 & 5 & 31 & 0 & pass \\
30 & 11 & 277 & 0 & pass \\
60 & 18 & 1101 & 0 & pass \\
100 & 26 & 3043 & 0 & pass \\
250 & 54 & 19023 & 0 & pass \\
500 & 96 & 76115 & 0 & pass \\
1000 & 169 & 304191 & 0 & pass \\
2000 & 304 & 1216587 & 0 & pass \\
5000 & 670 & 7600457 & 0 & pass \\
\bottomrule
\end{tabular}

\end{table}

\section{Exact Balanced Two-Color Endpoint}
\label{sec:balanced}

The lower-bound coloring in the proof is deliberately simple: put $1$ in the
first color, split the primes into bins, and color each composite by a bin
containing one of its prime divisors.  This canonical witness can be highly
unbalanced.  At $k=10$ and $n=60$, for example, it colors $51$ vertices one
way and only $9$ the other way; \Cref{tab:prime-bin-imbalance} records this
diagnostic over a wider range.

\begin{table}[htbp]
\centering
\caption{Imbalance of the canonical prime-bin coloring at $n=\Rcop(k;2)-1$.}
\label{tab:prime-bin-imbalance}
\begin{tabular}{rrrrr}
\toprule
$k$ & extremal $n$ & color sizes & imbalance & minority fraction \\
\midrule
2 & 2 & 1:1 & 0 & 0.500 \\
3 & 6 & 4:2 & 2 & 0.333 \\
4 & 12 & 9:3 & 6 & 0.250 \\
5 & 18 & 14:4 & 10 & 0.222 \\
6 & 28 & 23:5 & 18 & 0.179 \\
7 & 36 & 30:6 & 24 & 0.167 \\
8 & 42 & 35:7 & 28 & 0.167 \\
9 & 52 & 44:8 & 36 & 0.154 \\
10 & 60 & 51:9 & 42 & 0.150 \\
11 & 70 & 60:10 & 50 & 0.143 \\
12 & 78 & 67:11 & 56 & 0.141 \\
13 & 88 & 76:12 & 64 & 0.136 \\
\bottomrule
\end{tabular}

\end{table}

This raises a natural objection: perhaps the theorem is an artifact of
allowing very unbalanced color classes.  To test this, we solved the following
near-balanced mixed-integer linear programming (MILP) model for small $k$:
find a two-coloring with
$|\,|\chi^{-1}(0)|-|\chi^{-1}(1)|\,|\le 1$ and with no monochromatic
coprime $K_k$.  For every coprime $k$-clique $K$, the model adds
\[
  1 \le \sum_{v\in K} x_v \le k-1,
\]
plus the near-balance constraint.  These diagnostics, recorded in
\Cref{app:balanced}, suggested that composite vertices provide enough
flexibility to rebalance a prime-bin certificate without creating a
monochromatic coprime clique.  The decisive pattern is the deterministic split
\[
  B_0=\{3,5,\ldots,p_{k-1}\},\qquad
  B_1=\{2,p_k,p_{k+1},\ldots,p_{2k-3}\}
\]
and it has an elementary exact analysis.

Let $L_{\mathrm{bal}}(k;2)$ be the largest $n$ for which there exists a
two-coloring of $[n]$ with color classes differing by at most one and with no
monochromatic coprime $K_k$.

\begin{lemma}[Weak prime-index gaps]
\label{lem:prime-index-gaps}
For every $m\ge2$,
\[
  2p_m<p_{2m}<3p_m .
\]
\end{lemma}

\begin{proof}
We use this only in a weak form.  The displayed prime-counting estimates in
Dusart's preprint \cite[Theorem~6.9]{dusart2010} give
\[
  \frac{x}{\log x-1}<\pi(x)\quad (x\ge5393),
  \qquad
  \pi(x)<\frac{x}{\log x-1.1}\quad (x\ge60184).
\]
Put $x=p_m$.  If $x\ge60184$, then
\[
  \pi(2x)<\frac{2x}{\log(2x)-1.1}
  <\frac{2x}{\log x-1}<2\pi(x)=2m,
\]
so $p_{2m}>2p_m$.  Also
\[
  \pi(3x)>\frac{3x}{\log(3x)-1}
  >\frac{2x}{\log x-1.1}>2\pi(x)=2m,
\]
because $\log x>1.3+2\log 3$ in this range; hence $p_{2m}<3p_m$.
The remaining finite range $2\le m\le \pi(60184)=6076$ is checked exactly by
direct computation.
\end{proof}

\begin{theorem}[Exact balanced endpoint]
\label{thm:balanced-exact}
For every $k\ge2$,
\[
  L_{\mathrm{bal}}(k;2)=p_{2k-2}-1.
\]
Equivalently, the balanced upper transition point is $p_{2k-2}$, the same as
the unrestricted two-color threshold.
\end{theorem}

\begin{proof}
The upper bound follows from \Cref{thm:main}.  The case $k=2$ is immediate
(place $1$ and $2$ in opposite colors), so assume $k\ge3$ and put
$x=p_{2k-2}-1$.  Then $x$ is even and the primes at most $x$ are exactly
$p_1,\ldots,p_{2k-3}$.

\medskip\noindent\textit{The deterministic split.}
Let
\[
  B_0=\{p_2,\ldots,p_{k-1}\},\qquad
  B_1=\{p_1,p_k,\ldots,p_{2k-3}\},
\]
so $|B_0|=k-2$ and $|B_1|=k-1$.  We will produce a two-coloring with vertex
$1$ in color $0$ and with every other vertex $v$ assigned color $i$ having
at least one prime divisor in $B_i$.  Such a coloring is divisor-certified:
color $0$ has clique size at most $1+|B_0|=k-1$ (vertex $1$ plus injected
witness primes) and color $1$ has clique size at most $|B_1|=k-1$, so no
monochromatic coprime $K_k$ exists.

\medskip\noindent\textit{Vertices forced into color $0$.}
Write $\sigma(v)$ for the set of prime divisors of $v$, and define
\[
  F_0=\{1\}\cup\{v\in[2,x]:\sigma(v)\subseteq B_0\}.
\]
Vertices in $F_0$ must be colored $0$ to preserve the certificate.  Since
$B_0$ contains only odd primes, every $v\in F_0$ is odd.  Conversely, an odd
$v\in[1,x]$ lies in $F_0$ unless $v$ has a prime divisor $q\ge p_k$; if
additionally $v\ne q$, then $v$ has a second prime factor of size at least
$3$, so $v\ge 3q\ge 3p_k>p_{2k-2}=x+1$ by \Cref{lem:prime-index-gaps}, a
contradiction.  The odd vertices in $[1,x]$ outside $F_0$ are therefore
exactly the primes $p_k,\ldots,p_{2k-3}$, of which there are $k-2$.  Since
$[1,x]$ contains $x/2$ odd integers,
\[
  |F_0|=\frac{x}{2}-(k-2).
\]

\medskip\noindent\textit{Flexible vertices moved into color $0$.}
For $i=2,\ldots,k-1$, set $f_i=2p_i$.  By \Cref{lem:prime-index-gaps},
$2p_{k-1}<p_{2(k-1)}=p_{2k-2}=x+1$, so every $f_i\le x$.  The support
$\sigma(f_i)=\{2,p_i\}$ meets both bins: $2\in B_1$ and $p_i\in B_0$, so
$f_i$ may be colored either way without breaking the certificate.  Assign
these $k-2$ vertices to color $0$.

\medskip\noindent\textit{Assigning the remaining vertices.}
Color every remaining vertex with $1$.  Each such vertex either has support
in $B_1$ (forced into color $1$) or has support meeting $B_1$ (flexible,
not selected above); in both cases coloring it $1$ uses a prime in $B_1$,
so the certificate is preserved.

\medskip\noindent\textit{Balance.}
Color $0$ now has $|F_0|+(k-2)=x/2$ vertices and color $1$ has the remaining
$x/2$.  The two-coloring is therefore exactly balanced and avoids
monochromatic coprime $K_k$ in both colors.
\end{proof}

Thus the balanced two-color objection is resolved exactly for every $k\ge2$:
the same divisor-certificate family contains a perfectly balanced extremal
coloring at every two-color diagonal endpoint, and the only number-theoretic
input is the weak prime-index gap of \Cref{lem:prime-index-gaps}.

\begin{corollary}[A density window around balance]
\label{cor:density-window}
Let $k\ge3$ and set $x=p_{2k-2}-1$.  For every integer $r$ satisfying
\[
  \left|r-\frac{x}{2}\right|\le k-2,
\]
there is a two-coloring of $[x]$ with exactly $r$ vertices in color $0$ and no
monochromatic coprime $K_k$.
\end{corollary}

\begin{proof}
Use the deterministic split from the proof of \Cref{thm:balanced-exact}.  The
forced color-$0$ set has size $x/2-(k-2)$, and the vertices
$2p_2,\ldots,2p_{k-1}$ are flexible because each has one prime divisor in
each bin.  Coloring any $t$ of these flexible vertices with color $0$ realizes
every color-$0$ size in
\[
  \frac{x}{2}-(k-2),\ \frac{x}{2}-(k-3),\ldots,\frac{x}{2}.
\]
Interchanging the two color names realizes the symmetric sizes above $x/2$.
\end{proof}

The window in \Cref{cor:density-window} has relative width
$O(k/p_{2k-2})=O(1/\log k)$.  It is not a full prescribed-density theorem, but
it records a useful robustness fact: the exact balanced witness is not a
single isolated coloring.  \Cref{tab:density-window-verify} confirms
the construction for selected $k$, and the summary in
\Cref{tab:density-window-summary} verifies the formula through
$k=100{,}000$ while retaining explicit enumeration through $k=500$.

\begin{corollary}[Off-diagonal balanced endpoints]
\label{cor:balanced-off-diagonal}
Let $s,t\ge2$, and let $L_{\mathrm{bal}}(s,t)$ be the largest $n$ for which
there is a two-coloring of $[n]$ with color classes differing by at most one,
with no color-$0$ coprime $K_s$ and no color-$1$ coprime $K_t$.  Then
\[
  L_{\mathrm{bal}}(s,t)=p_{s+t-2}-1 .
\]
\end{corollary}

\begin{proof}
The upper bound is \Cref{thm:main}.  For the lower bound, set
$a=\max\{s,t\}$, $b=\min\{s,t\}$, $M=a+b-2$, and $x=p_M-1$.  Put vertex $1$
in the color with demand $a$, and split the primes at most $x$ as
\[
  B_A=\{p_2,\ldots,p_{a-1}\},\qquad
  B_B=\{p_1,p_a,\ldots,p_{M-1}\}.
\]
Then $|B_A|=a-2$ and $|B_B|=b-1$.  As in the proof of
\Cref{thm:balanced-exact}, every odd vertex not forced into color $A$ must be
one of the primes $p_a,\ldots,p_{M-1}$.  Indeed, if an odd composite has a
prime factor $q\ge p_a$, then it is at least $3q\ge3p_a>p_M=x+1$, since
$M\le2a-2<2a$ and \Cref{lem:prime-index-gaps} gives $p_{2a}<3p_a$.
Thus the forced color-$A$ set has size $x/2-(b-2)$.  The $b-2$ vertices
$2p_2,\ldots,2p_{b-1}$ lie in $[x]$: indeed
$2p_{b-1}<p_{2b-2}\le p_M=x+1$ by \Cref{lem:prime-index-gaps} and
$2b-2\le M$.  Each has one prime divisor in each bin, so these vertices are
flexible.  Assigning them to color $A$ gives exactly $x/2$ vertices in each
color.  The divisor witnesses bound the color-$A$ clique size by $a-1$ and
the color-$B$ clique size by $b-1$.  Relabeling the colors if necessary gives
the stated ordered pair $(s,t)$.
\Cref{tab:balanced-off-diagonal-grid} confirms the construction on the full
grid $2\le s,t\le1000$; \Cref{tab:balanced-off-diagonal-verify} lists selected
instances from that grid.
\end{proof}

\section{Boundary Cases}
\label{sec:boundary}

\subsection{Multicolor Balance}

A natural next guess is that the same balanced endpoint should hold for
$c\ge3$ colors at $n=p_{c(k-1)}-1$.  This is false in the first nontrivial
case.  An exact MILP check shows that there is no balanced $3$-coloring of
$[12]$ into three classes of size $4$ avoiding monochromatic coprime triples,
even though $12=p_6-1$ is the unrestricted extremal endpoint for
$\Rcop(3;3)=p_6=13$.  Thus the deterministic two-color split above is not a
minor variant of a general all-color balance theorem; multicolor balance is
a new density-constrained problem.

\begin{table}[htbp]
\centering
\caption{Selected multicolor balanced endpoint MILP checks.  The endpoint
is $n=p_{c(k-1)}-1$; ``no'' means every balanced $c$-coloring at that
endpoint already contains a monochromatic coprime $K_k$, and ``unknown'' marks
a time- or resource-limited instance.}
\label{tab:balanced-multicolor-endpoint}
\small
\begin{tabular}{@{}rrrrr@{}}
\toprule
$c$ & $k$ & endpoint $n$ & coprime $K_k$ count & balanced feasible? \\
\midrule
3 & 3 & 12 & 79 & no \\
4 & 3 & 18 & 277 & no \\
5 & 3 & 28 & 1016 & no \\
3 & 4 & 22 & 928 & yes \\
3 & 5 & 36 & 14767 & yes \\
4 & 4 & 36 & 6979 & yes \\
6 & 3 & 36 & 2150 & no \\
7 & 3 & 42 & 3522 & no \\
4 & 5 & 52 & -- & unknown \\
\bottomrule
\end{tabular}

\end{table}

The obstruction is finite rather than asymptotic.  A round-robin prime-bin
partition gives balanced certificate colorings after a small initial range.
In the original scan, this begins at $k=6,8,15,16,24,28,37,53$ for
$c=3,\ldots,10$, respectively, through $k=1000$.  A subsequent phase scan
using a lower/upper bounded max-flow assignment extends the same start-$0$
round-robin strategy to $3\le c\le20$ through $k=500$ and to
$21\le c\le30$ through $k=400$; the largest observed onset in that scan is
$k=373$ for $c=30$.  Testing all round-robin starts for $3\le c\le20$ and
$k\le250$ did not improve these onsets, so the simple start-$0$ rule is not
being hidden by a better cyclic shift.  See \Cref{app:balanced} for the
certificate-family diagnostic,
\Cref{tab:balanced-multicolor-defect-map} for the exact small solver statuses
that were computationally decidable, and \Cref{tab:balanced-multicolor-phase}
for the extended phase summary.

\begin{theorem}[Eventual multicolor balanced certificates]
\label{thm:eventual-multicolor-balanced}
Fix $c\ge2$.  For all sufficiently large $k$, there is a balanced
$c$-coloring of the endpoint $[p_{c(k-1)}-1]$ with no monochromatic coprime
$K_k$.  Consequently, if $L_{\mathrm{bal}}(k;c)$ denotes the largest such
balanced avoiding endpoint, then
\[
  L_{\mathrm{bal}}(k;c)=p_{c(k-1)}-1
\]
for all sufficiently large $k$.
\end{theorem}

\begin{proof}
The case $c=2$ is \Cref{thm:balanced-exact}, so assume $c\ge3$ is fixed.
Put $M=c(k-1)$ and $x=p_M-1$.  The primes at most $x$ are
$p_1,\ldots,p_{M-1}$.  Choose uniformly a partition of these primes into
bins $B_1,\ldots,B_c$ with
\[
  |B_1|=k-2,\qquad |B_i|=k-1\quad(2\le i\le c),
\]
and put vertex $1$ in color $1$.  Set $A(1)=\{1\}$.  For $v>1$, let
\[
  A(v)=\{i:\sigma(v)\cap B_i\ne\emptyset\}
\]
be the set of colors allowed by the divisor certificate.  It is enough to
show that, with positive probability, the vertices can be assigned to allowed
colors with a fixed balanced target vector
$t_1,\ldots,t_c\in\{\lfloor x/c\rfloor,\lceil x/c\rceil\}$ satisfying
$\sum_i t_i=x$.

By the max-flow form of Hall's theorem, this assignment exists if for every
proper nonempty subset $S\subsetneq[c]$ the number
\[
  N_S=\#\{v\in[x]:A(v)\subseteq S\}
\]
is at most the total target capacity of the colors in $S$.  Fix such an $S$,
write $s=|S|$, and let $U_S=\bigcup_{i\in S}B_i$.  Since $c$ and $s<c$ are
fixed,
\[
  \frac{|U_S|}{M-1}\le z_s<1
\]
for all sufficiently large $k$, with $z_s$ depending only on $c$ and $s$.
For a fixed integer $v>1$ with $\omega(v)$ distinct prime divisors, the
probability that all of its prime divisors fall in $U_S$ is at most
$z_s^{\omega(v)}$: the prime bins are sampled with fixed capacities, so the
distinct prime divisors are exposed without replacement, and their indicators
are negatively associated; equivalently, each conditional exposure has
probability at most $|U_S|/(M-1)\le z_s$.  Hence
\[
  \mathbb E N_S
  \le 1+\sum_{v\le x} z_s^{\omega(v)}
  =O_{c,s}\!\left(x(\log x)^{z_s-1}\right)
  =o(x),
\]
where the middle estimate is the standard fixed-$z<1$ Selberg--Delange bound
for $\sum_{v\le x}z^{\omega(v)}$ \cite[Chapter II.5]{tenenbaum2015}.
Markov's inequality and a union bound over the finitely many proper
$S\subsetneq[c]$ show that, for some prime-bin partition and all sufficiently
large $k$, every $N_S$ is at most $s x/(2c)$.  Since the balanced target
capacity of $S$ is at least $s\lfloor x/c\rfloor$, this is below
$\sum_{i\in S}t_i$ for all large $x$.

Choose a balanced assignment satisfying these capacities.  If a vertex is
assigned color $i$, then it has a divisor in $B_i$, so the same
prime-divisor injection used in \Cref{thm:main} bounds monochromatic coprime
cliques.  Color $1$ has at most $1+|B_1|=k-1$ vertices in any such clique,
and every other color has at most $|B_i|=k-1$.  Thus the balanced assignment
avoids monochromatic coprime $K_k$ at $x=p_{c(k-1)}-1$.  The upper endpoint is
the unrestricted value from \Cref{thm:main}.
\end{proof}

Thus multicolor balance has a small-value defect followed by an eventual
certificate regime for every fixed $c$, rather than a uniform exact endpoint
theorem valid from $k=2$ onward.

\subsection{Shifted Intervals}

The exact theorem depends on two features of $[n]$: the universal vertex
$1$ and the initial prime clique.  To test the boundary of the result, we
computed the two-color threshold on shifted intervals
\[
  I_{m,n}=\{m+1,\ldots,m+n\}
\]
for $m=2,\ldots,50$ and $k=3,4,5$.  The induced graph has the same coprime
adjacency rule, but the vertex set no longer contains $1$ and the prime
clique is interval-dependent.  Each entry was solved by an exact binary MILP
over all coprime $K_k$ subgraphs in the interval.
An expanded deterministic boundary scan reaches all shifts $2\le m\le100$ for
$k=3,4,5$ exactly; at $k=6$, the same formulation already produces many
time-limited instances, so we treat that as the practical exact-MILP frontier
rather than as primary evidence for a formula.

\begin{proposition}[A prime-clique upper bound for shifted intervals]
\label{prop:shifted-upper}
For an integer $m\ge0$, let $R_{\mathrm{cop}}^{(m)}(k;2)$ be the least $n$ such that
every two-coloring of $\{m+1,\ldots,m+n\}$ contains a monochromatic coprime
$K_k$.  With the convention $\pi(0)=0$,
\[
  R_{\mathrm{cop}}^{(m)}(k;2)
  \le
  \min\{n:\pi(m+n)-\pi(m)\ge 2k-1\}
  \le p_{\pi(m)+2k-1}-m .
\]
\end{proposition}

\begin{proof}
If $\pi(m+n)-\pi(m)\ge 2k-1$, then the primes in $(m,m+n]$ are vertices of
the interval and form a coprime clique of size at least $2k-1$.  In any
two-coloring, one color appears on at least $k$ vertices of this clique.  The
second displayed bound is obtained by taking
$n=p_{\pi(m)+2k-1}-m$.
\end{proof}

The bound is deliberately one-sided.  It uses a clique of $2k-1$ interval
primes; without the universal vertex $1$, a clique of size $2k-2$ can split
evenly between the two colors.  For $m=0$ this gives the weaker upper bound
$p_{2k-1}$, while \Cref{thm:main} improves it to $p_{2k-2}$ exactly because
the initial interval also contains the universal vertex.  Thus shifted
intervals isolate the role played by vertex $1$ in the main theorem.
A systematic attempt to adapt the prime-bin construction to shifted intervals
as explicit lower-bound certificates found no applicable cases in the tested
range.  In the expanded scan summarized in
\Cref{tab:shifted-lower-bound-summary}, this happened for all
$499\cdot5=2495$ parameter choices with $2\le m\le500$ and $3\le k\le7$
under the tested interval lengths.  Without the universal vertex and the full
initial prime clique, the naive interval-adapted bin construction typically
leaves many integers without a witness prime.  This negative result
reinforces the conclusion that the certificate is genuinely sensitive to the
presence of vertex $1$ and the initial-segment structure.

\begin{table}[htbp]
\centering
\caption{Selected shifted-interval thresholds.  The entry is the least
length $n$ such that every two-coloring of $\{m+1,\ldots,m+n\}$ contains a
monochromatic coprime $K_k$.}
\label{tab:shifted-intervals}
\small
\begin{tabular}{@{}rrrr@{}}
\toprule
shift $m$ & $k=3$ & $k=4$ & $k=5$ \\
\midrule
2 & 9 & 15 & 21 \\
3 & 8 & 14 & 20 \\
5 & 8 & 14 & 20 \\
10 & 7 & 13 & 19 \\
20 & 9 & 17 & 23 \\
30 & 7 & 13 & 23 \\
40 & 7 & 13 & 19 \\
50 & 9 & 17 & 23 \\
\bottomrule
\end{tabular}

\end{table}

\begin{figure}[htbp]
\centering
\includegraphics[width=0.76\linewidth]{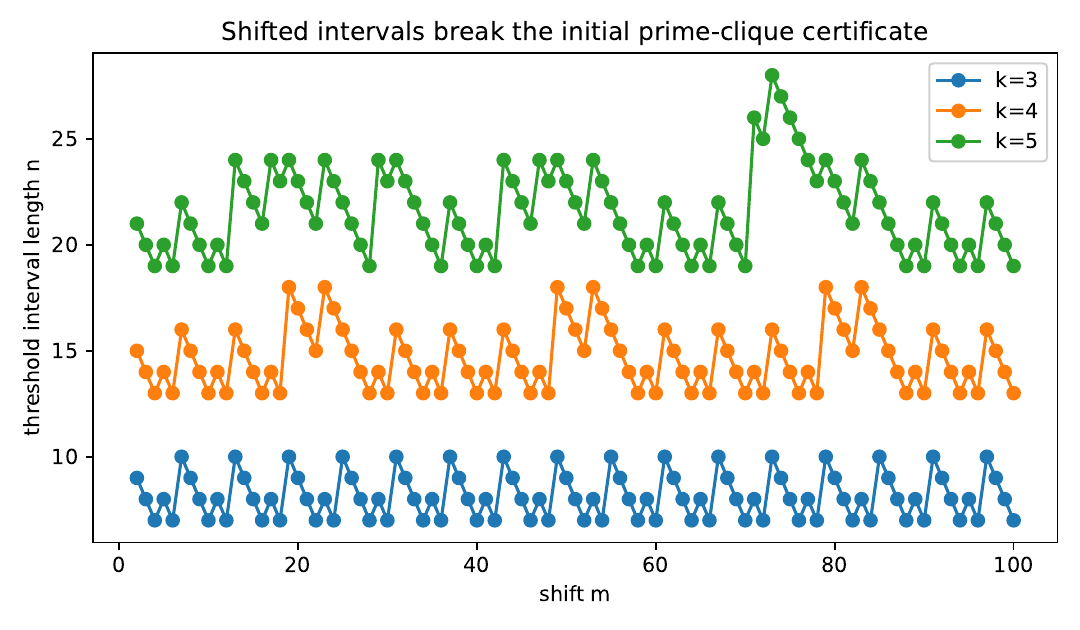}
\caption{Shifted intervals produce a small but genuine dependence on local
arithmetic structure.  The prime-bin proof for $[n]$ does not transfer
verbatim once vertex $1$ is removed.}
\label{fig:shifted-intervals}
\end{figure}

The entries should be read as local certificate data rather than as a noisy
version of the unshifted formula.  For example, the equality at $m=10$ and
$k=3$ is witnessed by the five-vertex coprime clique
\[
  \{11,13,14,15,17\}\subseteq[11,17],
\]
which plays the same forcing role as the initial prime clique in $[7]$ even
though it is no longer made only of primes and the vertex $1$ is absent.  By
contrast, nearby shifts may need longer intervals before such a local
forcing structure appears.

\section{Relation to Existing Work}

\paragraph{Classical and graph Ramsey theory.}
The theorem sits beside, rather than inside, the classical edge-coloring
Ramsey-number program initiated by Ramsey, Erd\H{o}s--Szekeres, Erd\H{o}s,
and Greenwood--Gleason
\cite{ramsey1930,erdos-szekeres1935,erdos1947,greenwood-gleason1955}.
That program remains computationally and asymptotically difficult; see the
dynamic survey of Radziszowski for small values and bounds
\cite{radziszowski2026}.  In the off-diagonal case, the classical estimate
$\Rcl(3,k)=\Theta(k^2/\log k)$ follows from Ajtai et al. and Kim
\cite{ajtai-komlos-szemeredi1980,kim1995}; by \Cref{cor:r3k-gap}, the
vertex-coprime analogue is smaller by a factor
$\Theta((\log k)^2/k)$.  The broader study of Ramsey properties of host
graphs goes back at least to Burr--Erd\H{o}s--Lov\'asz
\cite{burr-erdos-lovasz1976}.  Our host graph is highly structured, but the
coloring is on vertices rather than edges, which is why the prime-bin
certificate collapses the problem.

\paragraph{Coprime graphs and common-factor graphs.}
The graph $G_n$ belongs to a number-theoretic graph family studied from
several directions.  Erd\H{o}s and S\'ark\"ozy studied cycles in the
coprime graph of integers \cite{erdos-sarkozy}; Berkove and Brilleslyper
studied cliques and complete bipartite subgraphs on consecutive intervals
\cite{berkove-brilleslyper2022}; Batta and Hajdu recently studied universal
representation questions for common-factor graphs, the complement viewpoint
\cite{batta-hajdu2026}.  Jorf, Boudine, and Oukhtite studied the coprime
divisors graph $\Gamma_N$ on the proper divisors of one composite integer
$N$, computing coloring parameters and proving perfectness
\cite{jorf-boudine-oukhtite2024}.  Recent arXiv work also studies
structural and spectral properties of $G_n$ itself and related finite-group
coprime graphs \cite{banerjee2025,ma-zhai-gao-pan2025,ranjan-singh2025}.
These papers are not mixed Ramsey partition results, but they motivate the
same support-disjointness model.

\paragraph{The chromatic special case.}
The case $k=2$ recovers the chromatic threshold of the coprime graph:
every $c$-coloring has a monochromatic edge precisely when
$\chi(G_n)>c$, and $\chi(G_n)=\pi(n)+1$.  The standard proof colors each
composite by one of its prime divisors and uses the clique
$\{1\}\cup\{p\le n\}$ for optimality; an informal version appears in
\cite{lavrov-mse}.  Theorem~\ref{thm:main} is not merely a restatement of
this folklore fact.  It replaces single-prime color classes by bins of
capacity $k_i-1$, accounts for the universal vertex $1$ by the asymmetric
$k_1-2$ capacity, and handles arbitrary mixed demands in one formula.

\paragraph{Extremal sets without many pairwise coprime integers.}
There is also a substantial number-theoretic literature on large subsets of
$[n]$ with no $k+1$ pairwise coprime integers, originating in conjectures of
Erd\H{o}s and including work of Choi, Ahlswede--Khachatrian, Chen--Zhou, and
Kiss--S\'andor--Yang
\cite{erdos1962,choi1973,ahlswede-khachatrian1994,chen-zhou,kiss-sandor-yang}.
Those papers study extremal size of one subset.  The present problem is a
Ramsey partition problem: can all of $[n]$ be partitioned into color classes
whose pairwise-coprime packing numbers stay below prescribed thresholds?
Equivalently, the extremal-set line asks for
\[
  \max\{|A|:A\subseteq[n],\ \nu(A)<k\},
\]
whereas the Ramsey line asks whether $[n]$ can be partitioned into sets
$A_1,\ldots,A_c$ with $\nu(A_i)<k_i$ for every $i$.  The prime-bin
construction makes the partition problem exactly soluble, even though the
single-set extremal problem has a different flavor.

\paragraph{Edge-coloring coprime Ramsey numbers.}
Finally, the vertex-coloring problem should not be confused with the
edge-coloring coprime Ramsey problem.  Edge-coloring colors each coprime
pair independently and has a much larger apparent search space.  However,
\Cref{thm:edge-reduction} shows that this variant is exactly the classical
edge Ramsey problem pulled back through prime labels.  The values reported
in Towell's online AI-assisted computational exploration
\cite{towell-blog,towell-github},
namely $\Redge(3;2)=11$, $\Redge(3;3)=53$, and $\Redge(4;2)=59$, are exactly
the classical edge Ramsey numbers $R(3,3)=6$, $R(3,3,3)=17$, and
$R(4,4)=18$ viewed through the prime-index map $N\mapsto p_{N-1}$ of
\Cref{thm:edge-reduction}.  Thus the overlap is one of host graph and
terminology, not of the main vertex-coloring theorem.  The primality pattern
noted in that computational project is likewise not a separate arithmetic
phenomenon for clique edge-coprime values; it follows immediately from the
prime-index reduction.
The same audit also explains the nearest non-clique variants in that project:
asymmetric clique and Gallai-triangle entries are still complete-witness
transfers, while path and cycle entries are not expected to be prime-index
equalities because their witnesses can use composite vertices in ways not
controlled by the label clique; see \Cref{tab:towell-variant-positioning}.

\begin{table}[htbp]
\centering
\caption{Positioning Towell's edge-coloring variants against the
clique-label transfer.  Clique and complete-pattern targets are exact
prime-index transfers; non-complete targets retain only the prime-clique
upper-bound mechanism.}
\label{tab:towell-variant-positioning}
\scriptsize
\setlength{\tabcolsep}{3pt}
\begin{tabular}{@{}>{\raggedright\arraybackslash}p{0.21\linewidth}>{\raggedright\arraybackslash}p{0.13\linewidth}>{\raggedright\arraybackslash}p{0.23\linewidth}>{\raggedright\arraybackslash}p{0.17\linewidth}>{\raggedright\arraybackslash}p{0.09\linewidth}@{}}
\toprule
Towell variant & reported value & rank problem & prime-index prediction & status \\
\midrule
edge $K_3$, two colors & 11 & $R(3,3)=6$ & $p_{5}=11$ & exact \\
edge $K_4$, two colors & 59 & $R(4,4)=18$ & $p_{17}=59$ & exact \\
edge $K_3$, three colors & 53 & $R(3,3,3)=17$ & $p_{16}=53$ & exact \\
asymmetric $K_2/K_3$ & 3 & $R(2,3)=3$ & $p_{2}=3$ & exact \\
asymmetric $K_2/K_4$ & 5 & $R(2,4)=4$ & $p_{3}=5$ & exact \\
asymmetric $K_3/K_4$ & 19 & $R(3,4)=9$ & $p_{8}=19$ & exact \\
Gallai triangle, three colors & 29 & $gr_3(K_3)=11$ & $p_{10}=29$ & exact \\
gcd-$d$ edge triangles & $11d$ & $R(3,3)=6$ on scaled labels & $d\,p_5=11d$ & exact \\
monochromatic paths & $5,7,9,10,13,13$ & path Ramsey on prime clique & upper bound only & not tight \\
monochromatic cycles & $11,8,13,11$ & cycle Ramsey on prime clique & upper bound only & not tight \\
\bottomrule
\end{tabular}

\end{table}

\paragraph{Novelty boundary.}
The closest prior lines above account for the chromatic special case,
one-set extremal questions, divisor-graph coloring parameters, and
edge-coloring computations.  In the checked public sources, we did not find
the mixed vertex-coloring partition formula
\[
  \Rcop(k_1,\ldots,k_c)=p_{\sum_i(k_i-1)}
\]
on the integer coprime graph.  This novelty statement is deliberately modest:
we claim only that the mixed vertex-coloring formula was not present in the
checked public sources, not that it is absent from every private manuscript or
unindexed web page.  The companion literature notes record the database
queries used to audit the May 2026 comparison, while the mathematical
distinction from the closest indexed lines is the one summarized in
\Cref{tab:prior-positioning}.

\begin{table}[htbp]
\centering
\caption{Vertex-coloring and edge-coloring coprime Ramsey problems have the
same host graph but different combinatorial degrees of freedom.}
\label{tab:vertex-edge-comparison}
\small
\begin{tabular}{@{}>{\raggedright\arraybackslash}p{0.17\linewidth}>{\raggedright\arraybackslash}p{0.26\linewidth}>{\raggedright\arraybackslash}p{0.23\linewidth}>{\raggedright\arraybackslash}p{0.18\linewidth}@{}}
\toprule
Feature & Vertex-coloring version & Edge-coloring version & Consequence \\
\midrule
Objects colored & Integers $1,\ldots,n$ & Coprime pairs $\{a,b\}$ & Vertex version has one global choice per integer \\
Forbidden pattern & One color class contains $k$ pairwise coprime integers & One edge color contains all edges of a coprime $K_k$ & Closer to classical Ramsey search \\
Known values & $\Rcop(3;2)=7$, $\Rcop(4;2)=13$, $\Rcop(10;2)=61$ & $\Redge(3;2)=11$, $\Redge(4;2)=59$ & Same host, different object \\
Proof mechanism & Prime-bin theorem gives all mixed values & Prime-label pullback gives $\Redge=p_{\Rcl-1}$ & Both use the prime clique differently \\
Search role & SAT unnecessary after theorem & Equivalent to classical Ramsey search & Computation imports classical bounds \\
\bottomrule
\end{tabular}

\end{table}

\begin{table}[htbp]
\centering
\caption{Precise positioning against the closest prior lines.}
\label{tab:prior-positioning}
\small
\begin{tabular}{@{}>{\raggedright\arraybackslash}p{0.27\linewidth}>{\raggedright\arraybackslash}p{0.31\linewidth}>{\raggedright\arraybackslash}p{0.32\linewidth}@{}}
\toprule
Line of work & Object & Difference from this paper \\
\midrule
Coprime graph coloring & $G_n$ and divisor-type graphs & Covers the $k=2$ chromatic shadow, not mixed Ramsey thresholds \\
Extremal pairwise-coprime sets & One subset $A\subseteq[n]$ maximizing $|A|$ & Single-set extremal problem, not a full partition of $[n]$ \\
Coprime divisor graph $\Gamma_N$ & Proper divisors of one composite $N$ & Same prime-partition motif, different host graph and coloring question \\
Edge-coprime Ramsey computations & Edge colorings of $G_n$ & Exact reduction here shows the values are classical Ramsey values in prime index \\
Classical Ramsey search & Edge colorings of $K_N$ & Supplies the edge-coprime input; vertex-coprime collapses further \\
\bottomrule
\end{tabular}

\end{table}

\section{What Remains Nontrivial}

The vertex-coloring coprime Ramsey numbers themselves are complete after
Theorem~\ref{thm:main}.  The clique edge-coloring variant is also
structurally resolved by \Cref{thm:edge-reduction}: any remaining numerical
uncertainty is exactly the uncertainty in the corresponding classical Ramsey
number.  Further value comes from variants where neither the one-universal
support theorem nor the complete-witness edge reduction applies directly:

\begin{problem}[Classical-to-coprime edge bounds]
Translate the best known classical Ramsey bounds into edge-coprime bounds.
For instance, the current $43\le R(5,5)\le 46$ window
\cite{radziszowski2026,angeltveit-mckay2024} gives
$181\le \Redge(5;2)\le 197$.
\end{problem}

\begin{problem}[Non-complete edge substructures]
Determine exact thresholds for edge-coloring targets such as paths, cycles,
trees, and complete bipartite graphs in $G_n$.  \Cref{prop:substructure-transfer}
gives the prime-clique upper bound, but the lower-bound pullback is no longer
valid because labels need not be injective on non-adjacent vertices of a
target.  The non-prime path and cycle values in
\Cref{tab:towell-variant-positioning} are concrete evidence that this is a
different regime.
\end{problem}

\begin{problem}[Density-constrained vertex colorings]
The exact balanced two-color diagonal endpoint is closed by
\Cref{thm:balanced-exact}, and the two-color off-diagonal endpoint is closed
by \Cref{cor:balanced-off-diagonal}.  However
\Cref{tab:balanced-multicolor-endpoint} shows that the same endpoint can fail
for small multicolor balance, while
\Cref{thm:eventual-multicolor-balanced} shows that every fixed color count has
an eventual exact endpoint.  Meanwhile \Cref{cor:density-window} gives a
certified two-color density window of relative width $O(1/\log k)$ around
$1/2$.  Determine the finite multicolor transition thresholds and the
prescribed-density transition for two colors away from this window.
\end{problem}

\begin{problem}[Intervals and shifted coprime graphs]
Replace $[n]$ by intervals $\{m+1,\ldots,m+n\}$.  The vertex $1$ and the
initial prime clique disappear, and the computations in
\Cref{tab:shifted-intervals,fig:shifted-intervals}, together with the
one-sided bound in \Cref{prop:shifted-upper}, show a genuine dependence on
the local distribution of prime factors.
\end{problem}

\begin{problem}[Other arithmetic graphs]
Study analogous vertex-coloring Ramsey thresholds for graphs defined by
conditions such as $\gcd(a,b)\in D$, squarefree kernels, or coprimality in
rings of integers.  The squarefree-kernel case itself is covered by
\Cref{cor:squarefree-kernel}; the open direction is to classify variants that
add adjacencies not explained by disjoint supports.
\end{problem}

\begin{problem}[Certificate-aware search tools]
Develop Ramsey-search tools that test for support or label certificates before
expanding a large SAT or MILP formulation.  In the coprime graph, this means
extracting prime-divisor supports, checking the clique-label rank, and reducing
the instance to bin packing, classical Ramsey data, or a flow problem whenever
the corresponding certificate applies.  Only the residual cases---for example
non-complete edge targets, shifted intervals, and prescribed-density
constraints outside the certificate window---should be handed to a brute-force
solver.  Such a pipeline would turn the failure mode of the direct SAT
encoding into a diagnostic: a hard instance is one whose obstruction survives
after the support-disjointness and clique-label reductions have been removed.
\end{problem}

\Cref{tab:variant-roadmap} condenses these open directions by identifying
which part of the certificate fails in each variant.

\begin{table}[htbp]
\centering
\caption{Research roadmap after the exact vertex-coloring theorem.}
\label{tab:variant-roadmap}
\small
\begin{tabular}{@{}>{\raggedright\arraybackslash}p{0.21\linewidth}>{\raggedright\arraybackslash}p{0.34\linewidth}>{\raggedright\arraybackslash}p{0.35\linewidth}@{}}
\toprule
Variant & Why the current proof may fail & Suggested next step \\
\midrule
Edge coloring of $G_n$ & Vertex prime-bin proof does not apply directly & Import classical Ramsey bounds through the prime-index map \\
Shifted intervals $\{m+1,\ldots,m+n\}$ & Vertex $1$ and the initial prime clique disappear & Study local prime-factor cliques and interval-dependent thresholds \\
Density-constrained vertex coloring & The exact half-balanced two-color endpoint is solved, but other densities need not align with the skip-2 split & Test mixed, multicolor, and prescribed-density variants \\
Other gcd graphs & Prime-divisor injection may not match clique structure exactly & Identify divisor certificates for each graph family \\
General host graphs with arithmetic labels & The prime-bin proof depends on complete coprime adjacency inside bins & Determine which label systems preserve exact partition thresholds \\
Certificate-aware search tools & Direct SAT hides support and label structure & Preprocess by atoms, clique-label rank, prime-bin capacity, and flow feasibility before exhaustive search \\
\bottomrule
\end{tabular}

\end{table}

\section{Conclusion}

The main theorem gives a complete mixed multicolor solution for vertex
colorings of the integer coprime graph:
\[
  \Rcop(k_1,\ldots,k_c)=p_{\sum_i(k_i-1)}.
\]
The proof has one object doing both jobs.  The prime clique forces the upper
bound, and the prime-bin coloring supplies the matching lower-bound
certificate.  What initially looks like a large SAT frontier is therefore a
prime-index threshold.

The support-disjointness theorem explains why the proof works.  The real
resource is not the interval $[n]$ itself, but an atom system in which cliques
inject into disjoint supports.  The certificate primitive makes this
structural statement algorithmic: given supports, it checks whether the
Ramsey instance belongs to this class and then returns either the forcing
clique or the avoiding coloring.  The squarefree-kernel formulation is a
formal instance of the same support theorem.

The clique edge-coloring variant behaves differently but is also resolved
structurally.  It does not produce new arithmetic Ramsey constants:
\[
  \Redge(k_1,\ldots,k_c)=p_{\Rcl(k_1,\ldots,k_c)-1}.
\]
Thus the clique edge problem inherits the classical complete-graph Ramsey
table exactly, including both its known values and its open numerical gaps.
Non-complete edge targets such as paths and cycles lie outside this exact
pullback and remain separate local-structure questions.

The balanced-coloring objection is closed as well.  The canonical prime-bin
coloring can be very unbalanced, but the deterministic split
$\{3,5,\ldots,p_{k-1}\}$ versus $\{2,p_k,\ldots,p_{2k-3}\}$ leaves exactly
$k-2$ flexible composites $2p_2,\ldots,2p_{k-1}$ available for rebalancing,
and the elementary inequality $2p_m<p_{2m}<3p_m$ for $m\ge2$ is enough to
count both the forced and the flexible sides exactly.  Thus the balanced
two-color diagonal threshold is not an additional obstruction; it is the
same prime-index threshold $p_{2k-2}$.  The same construction also realizes a
small density window around $1/2$, and it extends to all two-color
off-diagonal demands, so exact balance is not an isolated diagonal point.
This exact balance phenomenon is also sharply scoped: for three or more
colors the same endpoint can already fail at $k=3$, while a Hall-theoretic
prime-bin argument proves that the endpoint becomes exact again for every
fixed color count once $k$ is sufficiently large.  Thus multicolor balance is a
genuine density-constrained variant with finite defects and an eventual
certificate regime, rather than a formal corollary of the two-color
construction.

The same formula also survives a small change from partitions to covers:
covering the prime clique still requires total capacity
\(\sum_i(k_i-1)\).  On the other hand, changing the host itself can destroy
the certificate.  Shifted intervals retain a prime-clique upper bound, but the
naive interval-prime lower-bound construction fails throughout the expanded
scan, which is a useful warning about overextending the formula.

The broader lesson is representation-dependent.  Some Ramsey problems are
genuinely governed by random, pseudorandom, SAT-based, or large language
model (LLM)-assisted search.  In this arithmetic host graph, the right
representation collapses the search to a small prime certificate.  The most
interesting next problems are therefore those that deliberately break this
certificate, such as shifted intervals or arithmetic adjacency rules where
support-disjointness no longer explains all edges.  Even there, the
prime-clique upper bound gives a first anchor point for the shifted-interval
thresholds.

This suggests a tool-building direction rather than only a list of open
values.  A useful arithmetic-Ramsey solver should first search for a compact
representation: support atoms, clique-label rank, prime-bin capacity, and
flow feasibility under density constraints.  SAT or MILP search should then
be applied to the residual instance after those certificate-controlled parts
have been collapsed.  In this sense, a failed direct SAT search is not just a
computational obstacle; it is a signal to look for the representation that
the encoding has hidden.

\appendix

\section{Proof Architecture}
\label{app:proof-anatomy}

The proof of Theorem~\ref{thm:main} is short because all obstruction and all
construction pass through the same object: prime divisors.  This appendix
spells out the mechanism in a way that is useful for checking variants.

\subsection{Pairwise-coprime packing}

For $A\subseteq[n]$, write
\[
  \nu(A)=\max\{|S|:S\subseteq A\text{ and }S\text{ is pairwise coprime}\}.
\]
The vertex-coloring problem asks whether every partition
$[n]=A_1\cup\cdots\cup A_c$ has $\nu(A_i)\ge k_i$ for some $i$.
For a set $A$ not containing $1$, every pairwise-coprime subset of $A$
injects into the set of prime divisors used by $A$: choose one prime divisor
from each selected integer.  Pairwise coprimality makes the choices distinct.
If $1\in A$, it can add exactly one extra element to such a packing.

Thus the lower-bound coloring is not a heuristic.  If every non-one integer
colored $i$ is assigned a witness prime in a bin $B_i$, then
\[
  \nu(A_i)\le |B_i|+\mathbf{1}_{1\in A_i}.
\]
The capacities in Theorem~\ref{thm:main} are exactly the largest bin sizes
that keep these quantities below $k_i$.

\subsection{\texorpdfstring{Why one color has capacity $k_1-2$}{Why one color has capacity k1-2}}

The asymmetry in the proof is only bookkeeping.  Vertex $1$ is coprime to
every other vertex, so whichever color receives $1$ has one unit of clique
capacity already spent.  We put $1$ in color $1$ and assign that color only
$k_1-2$ witness primes.  Any other color can receive $k_i-1$ witness primes.
Because the total number of primes below the threshold is at most
$\sum_i(k_i-1)-1$, the capacities exactly fit.

\subsection{Mixed examples}

The same formula simultaneously gives diagonal, off-diagonal, and multicolor
values.  The following examples are included mainly to prevent a common
misreading: there is no separate SAT frontier for the off-diagonal or
multicolor vertex problem.  \Cref{tab:mixed-examples} gives representative
values, and \Cref{fig:off-diagonal-heatmap} visualizes the prime-index
diagonals in the two-color off-diagonal case.

\begin{table}[htbp]
\centering
\caption{Sample mixed vertex-coprime Ramsey values from the exact formula.}
\label{tab:mixed-examples}
\begin{tabular}{llr}
\toprule
Demand type & Parameters & Exact value \\
\midrule
Diagonal, 3 colors & $\Rcop(3;3)=p_6$ & 13 \\
Diagonal, 4 colors & $\Rcop(3;4)=p_8$ & 19 \\
Diagonal, 5 colors & $\Rcop(3;5)=p_{10}$ & 29 \\
Diagonal, 6 colors & $\Rcop(3;6)=p_{12}$ & 37 \\
Off diagonal & $\Rcop(3,4)=p_5$ & 11 \\
Off diagonal & $\Rcop(3,5)=p_6$ & 13 \\
Off diagonal & $\Rcop(4,5)=p_7$ & 17 \\
Off diagonal & $\Rcop(5,7)=p_{10}$ & 29 \\
\bottomrule
\end{tabular}

\end{table}

\begin{figure}[htbp]
\centering
\includegraphics[width=0.66\linewidth]{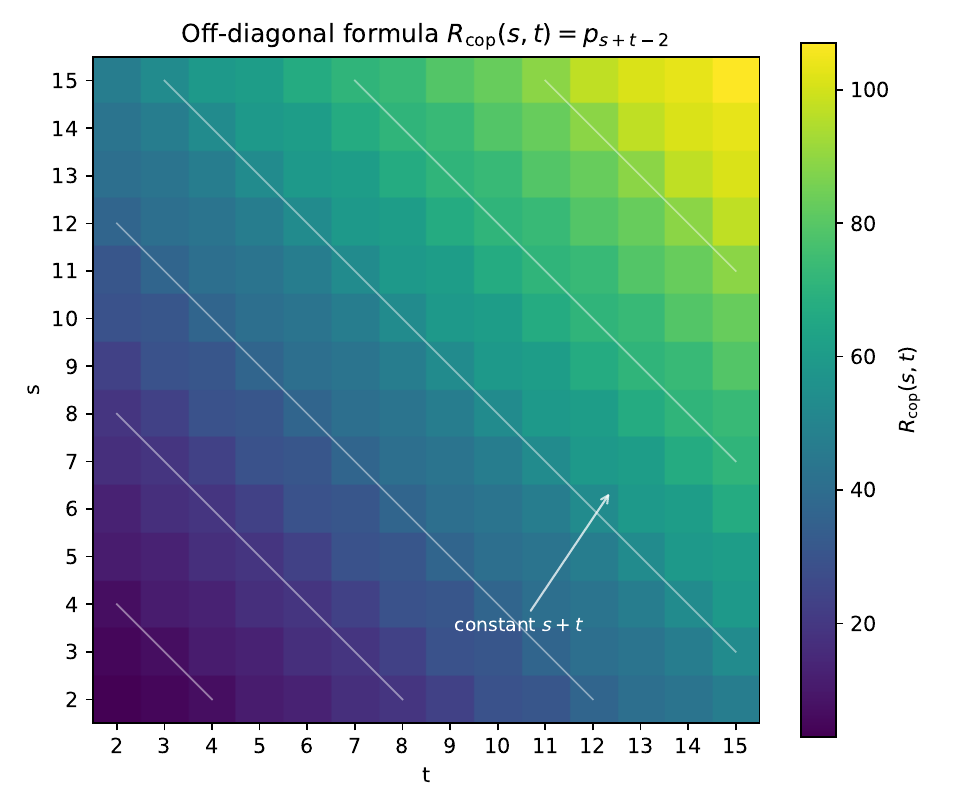}
\caption{Supplementary visualization of the off-diagonal formula: the values
form prime-index diagonals $\Rcop(s,t)=p_{s+t-2}$.}
\label{fig:off-diagonal-heatmap}
\end{figure}

\section{Computational Record, Literature Notes, and Supplementary Diagnostics}
\label{app:reproducibility}

The theorem is independent of computation, but the finite tables and figures
are reproducible.  The companion repository contains Python scripts and
generated data for the exact values, edge-reduction table, support-certificate
examples, SAT-scale diagnostics, balanced-coloring MILP and certificate-family
experiments, shifted-interval instances, and literature-search notes.  The
scripts use the standard library together with NumPy, SciPy (including the
HiGHS-backed MILP interface), pycosat for small conjunctive normal form (CNF)
checks, and Matplotlib.
The exact formula and
edge-reduction computations are deterministic; randomized balanced-witness
searches use a fixed seed and validate every reported witness against the
prime-bin certificate.  The README in the companion repository records the
full command list for regenerating the artifacts.  The largest certificate
scans reported here verify the density-window formula through $k=100{,}000$,
the off-diagonal balanced construction on the grid $2\le s,t\le1000$, the
multicolor certificate-family regime for $3\le c\le10$ and $k\le1000$, an
extended multicolor phase scan through $c=30$, and the shifted-interval
lower-bound diagnostic on $2495$ parameter choices.  The
shifted exact MILP scan is complete for $2\le m\le100$ and $k=3,4,5$; at
$k=6$, the same formulation already has resource-limited instances.
\Cref{tab:edge-bound-transfer} is included as a direct reference table for
the edge-coprime reduction, while
\Cref{tab:prime-index-check,tab:density-window-summary,tab:density-window-verify}
record the finite checks behind the balanced and density-window diagnostics.

\begin{table}[htbp]
\centering
\caption{Selected classical-to-edge-coprime bound translations.  The
classical two-color complete-graph bounds are taken from Radziszowski's
\emph{Small Ramsey Numbers}, Dynamic Survey DS1 revision 18, Tables Ia/Ib
\cite{radziszowski2026}; the
third column is the direct prime-index image
$p_{L-1}\le \Redge(k,\ell)\le p_{U-1}$ of each classical window
$L\le R(k,\ell)\le U$.}
\label{tab:edge-bound-transfer}
\small
\begin{tabular}{@{}lrr@{}}
\toprule
Classical parameter & $R(k,\ell)$ & translated $\Redge(k,\ell)$ \\
\midrule
R(3,3) & 6 & 11 \\
R(3,4) & 9 & 19 \\
R(3,5) & 14 & 41 \\
R(3,6) & 18 & 59 \\
R(3,7) & 23 & 79 \\
R(3,8) & 28 & 103 \\
R(3,9) & 36 & 149 \\
R(3,10) & 40--41 & 167--173 \\
R(4,4) & 18 & 59 \\
R(4,5) & 25 & 89 \\
R(4,6) & 36--40 & 149--167 \\
R(4,7) & 49--58 & 223--269 \\
R(4,8) & 59--79 & 271--397 \\
R(4,9) & 73--105 & 359--569 \\
R(4,10) & 92--135 & 467--757 \\
R(5,5) & 43--46 & 181--197 \\
R(5,6) & 59--85 & 271--433 \\
R(5,7) & 80--133 & 401--743 \\
R(5,8) & 101--193 & 541--1163 \\
R(5,9) & 133--282 & 743--1823 \\
R(5,10) & 149--381 & 857--2617 \\
R(6,6) & 102--160 & 547--937 \\
R(6,7) & 115--270 & 619--1723 \\
R(6,8) & 134--423 & 751--2917 \\
R(6,9) & 183--651 & 1091--4831 \\
R(6,10) & 204--944 & 1237--7451 \\
R(7,7) & 205--492 & 1249--3517 \\
R(7,8) & 219--832 & 1361--6373 \\
R(7,9) & 252--1368 & 1597--11311 \\
R(7,10) & 292--2119 & 1901--18493 \\
\bottomrule
\end{tabular}

\end{table}

\begin{table}[htbp]
\centering
\caption{Supplementary finite verification for the weak prime-index inequality
$2p_m<p_{2m}<3p_m$.  The proof only needs the finite range below Dusart's
explicit threshold, but the check was also extended well beyond it.}
\label{tab:prime-index-check}
\small
\begin{tabular}{@{}lrr@{}}
\toprule
finite check & value & attained at \\
\midrule
range checked & $2\le m\le 1000000$ & -- \\
$p_{2m}-2p_m$ minimum & 1 & $m=2$ \\
$3p_m-p_{2m}$ minimum & 2 & $m=2$ \\
\bottomrule
\end{tabular}

\end{table}

\begin{table}[htbp]
\centering
\caption{Large-scale verification summary for the density-window construction.
For $k>500$, the scan uses the closed-form counts from the proof and verifies
the prime-index inequalities needed for the flexible vertices.}
\label{tab:density-window-summary}
\small
\begin{tabular}{@{}rrrr@{}}
\toprule
$k$ range & enumerated through & rows checked & all verified \\
\midrule
3--100000 & 500 & 99998 & yes \\
\bottomrule
\end{tabular}

\end{table}

\begin{table}[htbp]
\centering
\caption{Verification of the density window construction
(\Cref{cor:density-window}).  $F_0$ matches the theoretical value
$n/2-(k-2)$ in every case, and the full window is realizable.}
\label{tab:density-window-verify}
\small
\begin{tabular}{@{}rrrrrrr@{}}
\toprule
$k$ & $n$ & $F_0$ (base) & $F_0$ (theorem) & window size & $F_0$ matches? & all realizable? \\
\midrule
3 & 6 & 2 & 2 & 3 & yes & yes \\
4 & 12 & 4 & 4 & 5 & yes & yes \\
5 & 18 & 6 & 6 & 7 & yes & yes \\
6 & 28 & 10 & 10 & 9 & yes & yes \\
7 & 36 & 13 & 13 & 11 & yes & yes \\
8 & 42 & 15 & 15 & 13 & yes & yes \\
9 & 52 & 19 & 19 & 15 & yes & yes \\
10 & 60 & 22 & 22 & 17 & yes & yes \\
15 & 106 & 40 & 40 & 27 & yes & yes \\
20 & 162 & 63 & 63 & 37 & yes & yes \\
30 & 270 & 107 & 107 & 57 & yes & yes \\
50 & 520 & 212 & 212 & 97 & yes & yes \\
100 & 1212 & 508 & 508 & 197 & yes & yes \\
200 & 2728 & 1166 & 1166 & 397 & yes & yes \\
500 & 7900 & 3452 & 3452 & 997 & yes & yes \\
1000 & 17382 & 7693 & 7693 & 1997 & yes & yes \\
10000 & 224716 & 102360 & 102360 & 19997 & yes & yes \\
100000 & 2750122 & 1275063 & 1275063 & 199997 & yes & yes \\
\bottomrule
\end{tabular}

\end{table}

\begin{table}[htbp]
\centering
\caption{Grid verification summary for the off-diagonal balanced construction
(\Cref{cor:balanced-off-diagonal}).}
\label{tab:balanced-off-diagonal-grid}
\small
\begin{tabular}{@{}rrrr@{}}
\toprule
$s,t$ range & pairs checked & max prime index & all verified \\
\midrule
2--1000 & 998001 & 1998 & yes \\
\bottomrule
\end{tabular}

\end{table}

\begin{table}[htbp]
\centering
\caption{Verification of the off-diagonal balanced construction
(\Cref{cor:balanced-off-diagonal}) for selected $(s,t)$.  The ``flexible''
column counts all available flexible composites in the interval; the proof
explicitly assigns only the $b-2$ composites $2p_2,\ldots,2p_{b-1}$ needed
to reach exact balance.}
\label{tab:balanced-off-diagonal-verify}
\small
\begin{tabular}{@{}rrrrrrr@{}}
\toprule
$s$ & $t$ & $n$ & $F_0$ & $F_1$ & flexible & balanced? \\
\midrule
3 & 4 & 10 & 4 & 4 & 2 & yes \\
3 & 10 & 30 & 5 & 14 & 11 & yes \\
10 & 30 & 162 & 15 & 73 & 74 & yes \\
50 & 50 & 520 & 212 & 63 & 245 & yes \\
100 & 150 & 1570 & 108 & 687 & 775 & yes \\
1000 & 1000 & 17382 & 7693 & 1095 & 8594 & yes \\
\bottomrule
\end{tabular}

\end{table}

\subsection{\texorpdfstring{SAT Encoding and the Superseded $R_{\rm cop}(10)=53$ Claim}{SAT Encoding and the Superseded Rcop(10)=53 Claim}}
\label{app:sat-forensics}

The original exploratory route treated the problem as a SAT search.  That
route was useful, but it also produced the main false lead.  We record the
correct encoding and the failure mode.

\subsubsection{Correct direct encoding}

For fixed $n,k,c$, introduce Boolean variables $x_{v,i}$ for
$v\in[n]$ and $i\in\{1,\ldots,c\}$, where $x_{v,i}$ means that vertex $v$
has color $i$.  The coloring constraints are
\[
  x_{v,1}\vee\cdots\vee x_{v,c}
\]
for every $v$, plus pairwise clauses
\[
  \neg x_{v,i}\vee \neg x_{v,j}\qquad (i<j)
\]
to force at most one color.  For every coprime $k$-clique
$K\subseteq[n]$ and every color $i$, add
\[
  \bigvee_{v\in K}\neg x_{v,i}.
\]
The formula is satisfiable exactly when there is a $c$-coloring of $G_n$
with no monochromatic coprime $K_k$.

\subsubsection{Why direct SAT looks hard}

At the exact threshold for $k=10$, the direct formula contains more than six
million coprime $K_{10}$ constraints and more than twelve million
anti-monochromatic clauses.  This is a real encoding explosion.  The theorem
shows that it is also avoidable: the unsatisfiability certificate only needs
the $19$-vertex prime clique at $n=61$.

\subsubsection{The invalid symmetry breaking}

The old computation for $R_{\rm cop}(10)$ tried to add prime-based symmetry
breaking by assigning the first $k-1$ primes distinct colors.  That operation
is sound only when the number of colors is at least $k-1$ and the intended
symmetry group permits that normalization.  It is not sound for a two-color
instance with $k=10$.

Worse, the old variable-indexing code accepted color indices outside the
range $\{0,1\}$.  A unit clause intended to mean ``prime $p$ has color $i$''
for $i>1$ was mapped to an unrelated SAT variable rather than rejected.  The
solver was therefore refuting a different formula.  The claim
$R_{\rm cop}(10;2)=53$ is superseded by the exact theorem and by the explicit
extremal coloring at $n=60$.

\subsubsection{Correct interpretation of the old frontier}

The value $53$ is meaningful, but for the adjacent problem:
\[
  \Rcop(9;2)=p_{16}=53.
\]
The correct two-color diagonal value for $k=10$ is
\[
  \Rcop(10;2)=p_{18}=61.
\]
At $n=60$, the prime-bin coloring avoids monochromatic $K_{10}$; at $n=61$,
the prime clique has $19$ vertices and pigeonhole forces ten in one color.
\Cref{fig:k10-coloring} shows the explicit extremal witness at $n=60$.

\begin{figure}[htbp]
\centering
\includegraphics[width=0.92\linewidth]{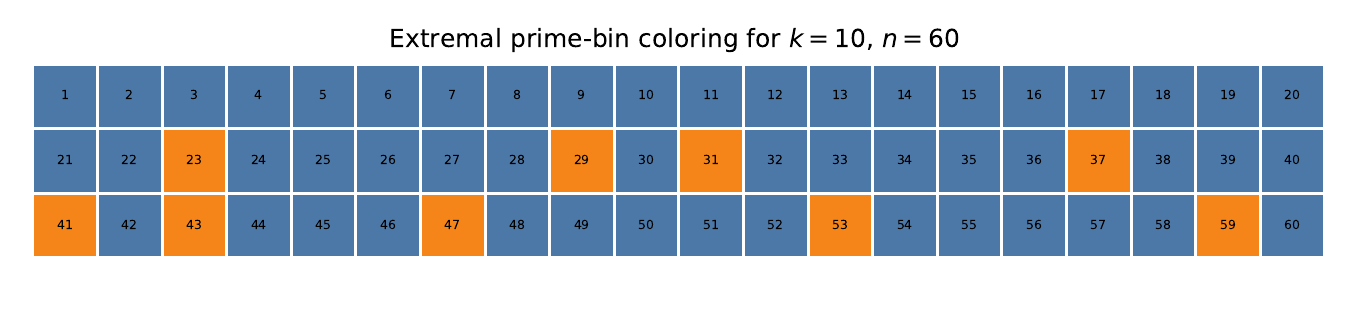}
\caption{One extremal prime-bin coloring at $k=10,n=60$.  Blue and orange
squares are the two color classes in a canonical prime-bin witness.  The
visible imbalance is diagnostic only; the exact balanced theorem is proved
independently in \Cref{thm:balanced-exact}.}
\label{fig:k10-coloring}
\end{figure}

\subsection{Balanced Colorings in More Detail}
\label{app:balanced}

Balanced coloring was the most useful side branch after the main theorem.
It tests whether the exact threshold is merely exploiting an obviously
unbalanced construction.  The answer is no exactly for the two-color
diagonal endpoint, and the experiments below show how the deterministic
proof was found.

\subsubsection{Exact MILP}

For the two-color balanced variant, variables $x_v\in\{0,1\}$ encode the
color of vertex $v$.  For every coprime $K_k$,
\[
  1 \le \sum_{v\in K} x_v \le k-1
\]
prevents monochromatic cliques in either color.  The near-balance constraint
is
\[
  \left|2\sum_{v=1}^n x_v-n\right|\le 1.
\]
Solving this model at $n=\Rcop(k;2)-1$ and using the unrestricted theorem
for the upper bound proves that the largest balanced avoiding endpoint is
the unrestricted endpoint whenever the MILP is feasible.  This gives
equality for $k=3,\ldots,9$; see
\Cref{tab:balanced-thresholds,fig:balanced}.

\begin{table}[htbp]
\centering
\caption{Near-balanced thresholds computed by exact MILP.}
\label{tab:balanced-thresholds}
\begin{tabular}{rrrrr}
\toprule
$k$ & balanced threshold & unrestricted & gap & last feasible \\
\midrule
3 & 7 & 7 & 0 & 6 \\
4 & 13 & 13 & 0 & 12 \\
5 & 19 & 19 & 0 & 18 \\
6 & 29 & 29 & 0 & 28 \\
7 & 37 & 37 & 0 & 36 \\
8 & 43 & 43 & 0 & 42 \\
9 & 53 & 53 & 0 & 52 \\
\bottomrule
\end{tabular}

\end{table}

\begin{figure}[htbp]
\centering
\includegraphics[width=0.74\linewidth]{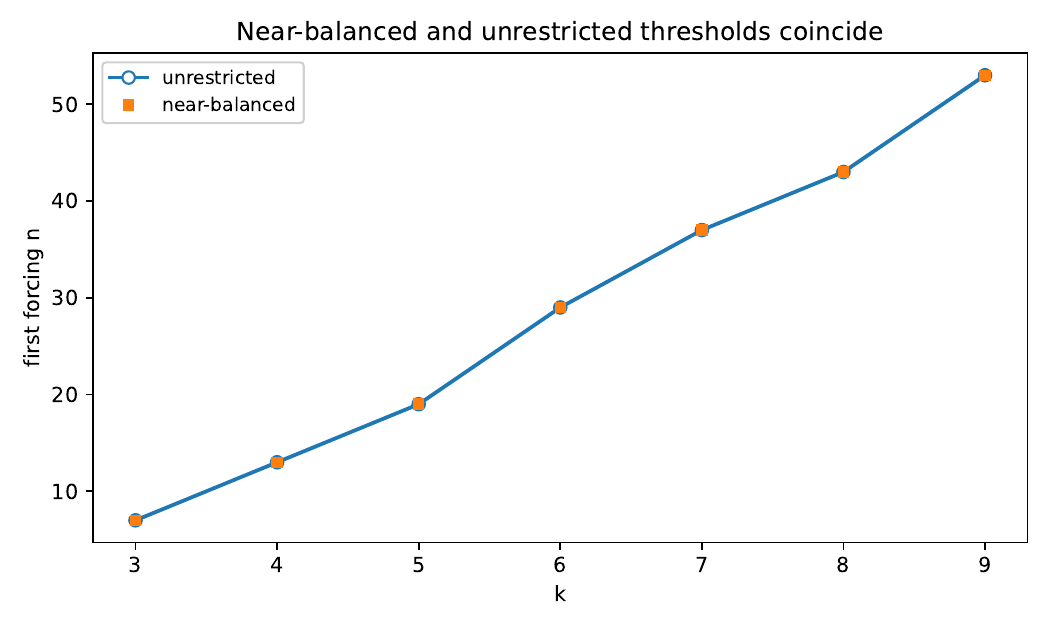}
\caption{For $k=3,\ldots,9$, near-balanced two-colorings exist up to the same
last feasible $n$ as unrestricted colorings.  The two sequences coincide; the
markers and line styles are chosen so both computed thresholds remain visible.}
\label{fig:balanced}
\end{figure}

\subsubsection{Constructive prime-bin witnesses}

The MILP treats every coprime clique explicitly.  A more structural search
stays inside the proof certificate.  At $n=p_{2k-2}-1$, there are $2k-3$
primes.  Choose $k-2$ of them for the bin containing vertex $1$ and put the
remaining $k-1$ primes in the other bin.  Integers whose prime divisors meet
only one bin are forced; integers whose prime divisors meet both bins are
flexible.

For $k=10$, one strict balanced witness uses
\[
  B_0=\{2,3,7,17,31,37,41,43\}
\]
and
\[
  B_1=\{5,11,13,19,23,29,47,53,59\}.
\]
At $n=60$, the set forced into color $0$ has size $30$, the set forced into
color $1$ has size $11$, and there are $19$ flexible composites.  Assigning
all flexible composites to color $1$ gives a $30{:}30$ coloring.  No
monochromatic $K_{10}$ exists: color $0$ has at most eight non-one witness
primes plus vertex $1$, and color $1$ has at most nine witness primes.

\Cref{thm:balanced-exact} proves the balanced diagonal endpoint exactly:
for every $k\ge2$, the largest $n$ admitting a balanced two-coloring of
$G_n$ with no monochromatic $K_k$ is $p_{2k-2}-1$.  The experiments below are
therefore no longer presented as evidence for an unresolved statement.  They
record how the deterministic construction was found and stress-test nearby
prime-bin choices.  The constructive witnesses in
\Cref{tab:balanced-prime-bin,fig:balanced-prime-bin} keep the search inside
the prime-bin certificate family.

\begin{table}[htbp]
\centering
\caption{Constructive near-balanced witnesses inside the prime-bin certificate family.}
\label{tab:balanced-prime-bin}
\begin{tabular}{rrrrrr}
\toprule
$k$ & $R-1$ & checked splits & color sizes & reachable color-0 interval & time (s) \\
\midrule
3 & 6 & 1 & 3:3 & [3, 4] & 0.000 \\
4 & 12 & 2 & 6:6 & [6, 8] & 0.000 \\
5 & 18 & 6 & 9:9 & [9, 13] & 0.000 \\
6 & 28 & 19 & 14:14 & [14, 22] & 0.000 \\
7 & 36 & 69 & 18:18 & [18, 28] & 0.000 \\
8 & 42 & 231 & 21:21 & [21, 33] & 0.001 \\
9 & 52 & 771 & 26:26 & [26, 43] & 0.004 \\
10 & 60 & 3,003 & 30:30 & [30, 49] & 0.016 \\
11 & 70 & 7,491 & 35:35 & [35, 59] & 0.045 \\
12 & 78 & 28,309 & 39:39 & [39, 65] & 0.185 \\
13 & 88 & 101,625 & 44:44 & [44, 74] & 0.794 \\
14 & 100 & 329,416 & 50:50 & [50, 85] & 2.743 \\
15 & 106 & 1,290,136 & 53:53 & [53, 90] & 11.421 \\
16 & 112 & 5,136,924 & 56:56 & [56, 95] & 48.627 \\
\bottomrule
\end{tabular}

\end{table}

\begin{figure}[htbp]
\centering
\includegraphics[width=0.74\linewidth]{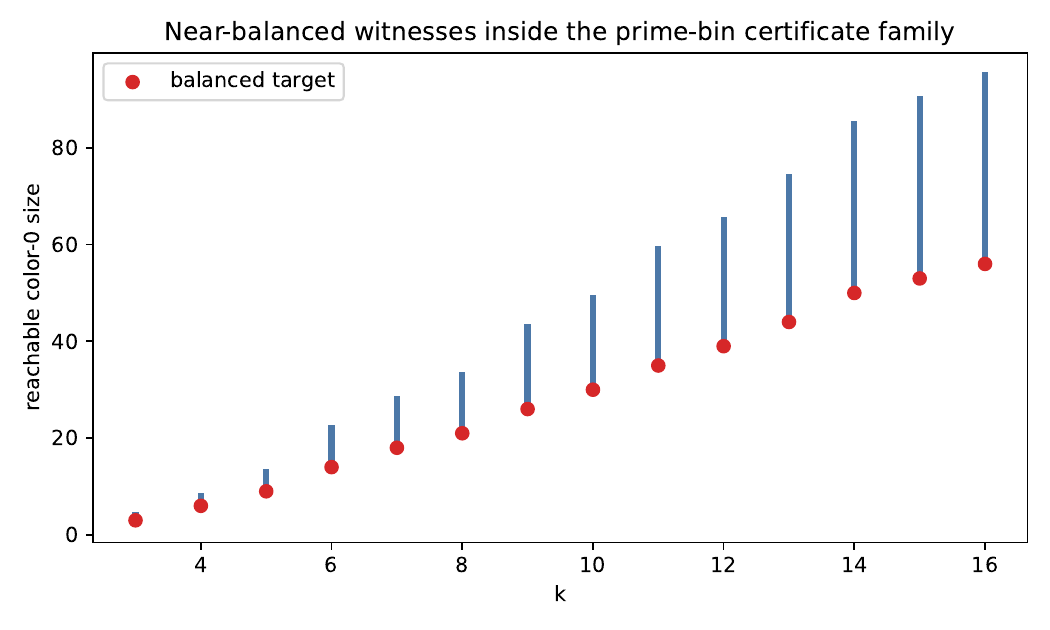}
\caption{For $k=3,\ldots,16$, the reachable color-0 interval inside one
prime-bin split contains the balanced target at $n=\Rcop(k;2)-1$.  Blue
segments show the reachable intervals and red markers show the balanced
targets.}
\label{fig:balanced-prime-bin}
\end{figure}

\paragraph{Randomized extension.}
The exhaustive search above enumerates prime-bin splits and is therefore
limited by $\binom{2k-3}{k-2}$.  For larger $k$, a randomized search instead
samples prime splits and validates successes by the same certificate.  With
a fixed random seed, it found near-balanced witnesses for every
$k=17,\ldots,80$ using at most three attempts per $k$ in the recorded run.

\begin{table}[htbp]
\centering
\caption{Sample rows from the randomized balanced prime-bin search.  The
search is heuristic, but each reported row is a validated certificate.}
\label{tab:balanced-random-prime-bin}
\small
\begin{tabular}{@{}rrrrr@{}}
\toprule
$k$ & $R-1$ & attempts & color sizes & reachable interval \\
\midrule
17 & 130 & 2 & 65:65 & [56, 106] \\
20 & 162 & 1 & 81:81 & [42, 117] \\
25 & 222 & 1 & 111:111 & [43, 142] \\
30 & 270 & 3 & 135:135 & [55, 174] \\
40 & 396 & 2 & 198:198 & [58, 238] \\
50 & 520 & 1 & 260:260 & [71, 290] \\
60 & 646 & 2 & 323:323 & [165, 504] \\
80 & 928 & 1 & 464:464 & [228, 729] \\
\bottomrule
\end{tabular}

\end{table}

Before the deterministic split was identified, random prime-bin choices
already showed that balanced witnesses were abundant inside the same
certificate family; \Cref{tab:balanced-random-prime-bin} records that
diagnostic evidence.

\paragraph{Deterministic bin choices.}
To look for a proof of the exact endpoint, we tested simple closed-form prime
splits up to $k=500$.  The canonical ``smallest primes'' choice is usually
bad, but three simple choices succeeded for every tested $k$: put
$B_0=\{3,5,\ldots,p_{k-1}\}$, or take alternating prime indices starting at
either parity.  The first of these is the split proved in
\Cref{thm:balanced-exact}; the alternating splits remain useful diagnostics.
\Cref{tab:balanced-deterministic,fig:balanced-deterministic-margins,fig:balanced-skip2-forced}
show the deterministic margins that led to the proof.

\begin{table}[htbp]
\centering
\caption{Deterministic prime-bin choices tested for balanced witnesses at
$n=p_{2k-2}-1$, for $k=3,\ldots,500$.}
\label{tab:balanced-deterministic}
\small
\begin{tabular}{@{}lrrrr@{}}
\toprule
strategy & successes & tested & first failure & worst $\max(F_0,F_1)/(n/2)$ \\
\midrule
smallest primes & 1 & 498 & 4 & 1.862 \\
largest primes & 0 & 498 & 3 & 1.862 \\
skip-2 then small & 498 & 498 & -- & 1.000 \\
alternating start 0 & 498 & 498 & -- & 1.000 \\
alternating start 1 & 498 & 498 & -- & 1.000 \\
middle block & 3 & 498 & 5 & 1.752 \\
\bottomrule
\end{tabular}

\end{table}

\begin{figure}[htbp]
\centering
\includegraphics[width=0.74\linewidth]{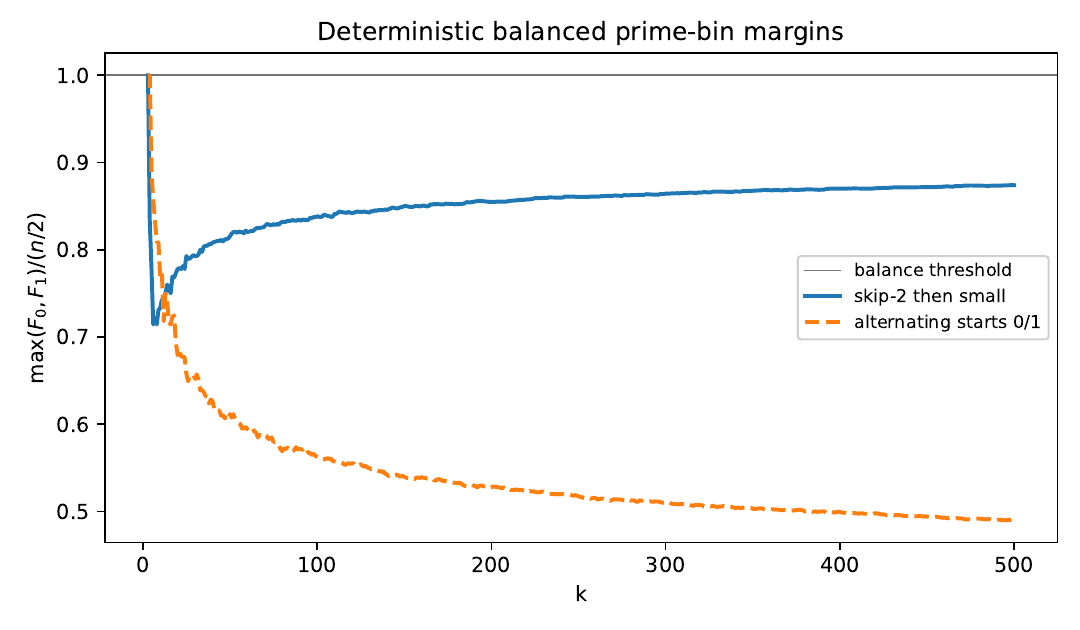}
\caption{For the deterministic splits that succeeded for
$k=3,\ldots,500$, the larger forced side stays at or below the balanced
threshold $n/2$; here $F_i$ is the set forced into color $i$ by the chosen
prime-bin split.  The two alternating starts have identical margins and are
shown as one dashed curve.  The skip-2 split is proved in
\Cref{thm:balanced-exact}.}
\label{fig:balanced-deterministic-margins}
\end{figure}

\begin{figure}[htbp]
\centering
\includegraphics[width=0.74\linewidth]{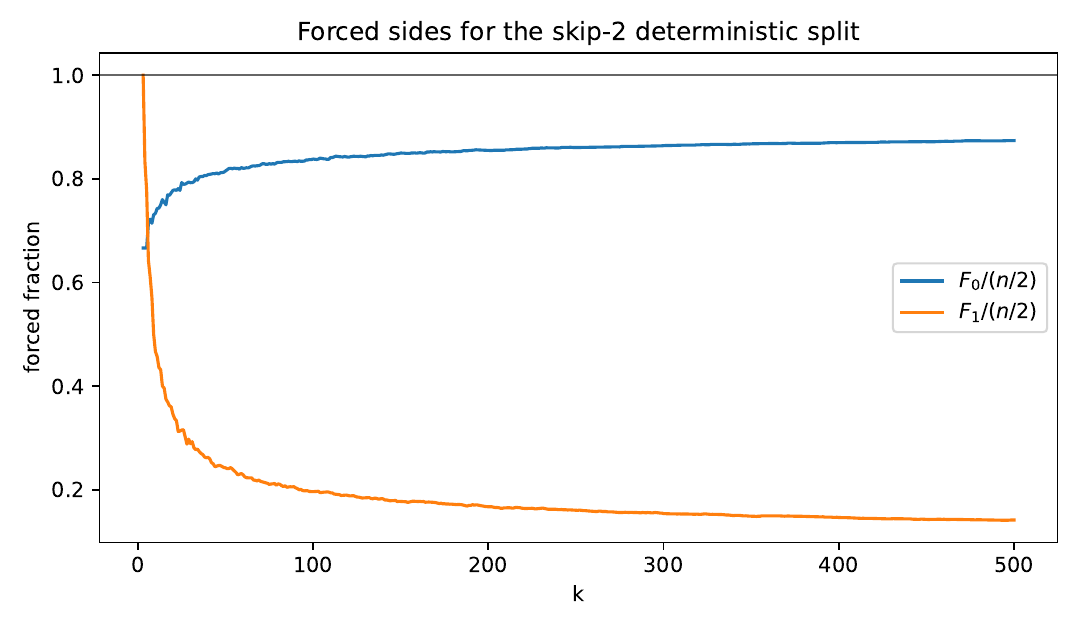}
\caption{For the proved skip-2 split, the color-0 forced side has the exact
formula $F_0=n/2-(k-2)$; the color-1 forced side is much smaller in the
tested range.  The remaining vertices are flexible and can be assigned
without losing the divisor certificate.}
\label{fig:balanced-skip2-forced}
\end{figure}

\paragraph{Multicolor certificate-family diagnostic.}
The exact endpoint failures in \Cref{tab:balanced-multicolor-endpoint} do
not mean that multicolor balance is structurally hopeless.  In the
prime-bin certificate family, round-robin prime partitions become feasible
after a small initial range for each tested color count $3\le c\le10$ in the
full $k\le1000$ scan; see
\Cref{tab:balanced-multicolor-summary,fig:balanced-multicolor-margins}.

\begin{table}[htbp]
\centering
\caption{Multicolor balanced searches inside the prime-bin certificate
family at $n=p_{c(k-1)}-1$, for $k=3,\ldots,1000$.  The displayed strategy is
the start-$0$ round-robin prime partition.}
\label{tab:balanced-multicolor-summary}
\small
\begin{tabular}{@{}rlrrrr@{}}
\toprule
$c$ & strategy & successes & tested & first failure & all-success from \\
\midrule
3 & round-robin start 0 & 995 & 998 & 3 & 6 \\
4 & round-robin start 0 & 993 & 998 & 3 & 8 \\
5 & round-robin start 0 & 986 & 998 & 3 & 15 \\
6 & round-robin start 0 & 985 & 998 & 3 & 16 \\
7 & round-robin start 0 & 977 & 998 & 3 & 24 \\
8 & round-robin start 0 & 973 & 998 & 3 & 28 \\
9 & round-robin start 0 & 964 & 998 & 3 & 37 \\
10 & round-robin start 0 & 948 & 998 & 3 & 53 \\
\bottomrule
\end{tabular}

\end{table}

\begin{figure}[htbp]
\centering
\includegraphics[width=0.74\linewidth]{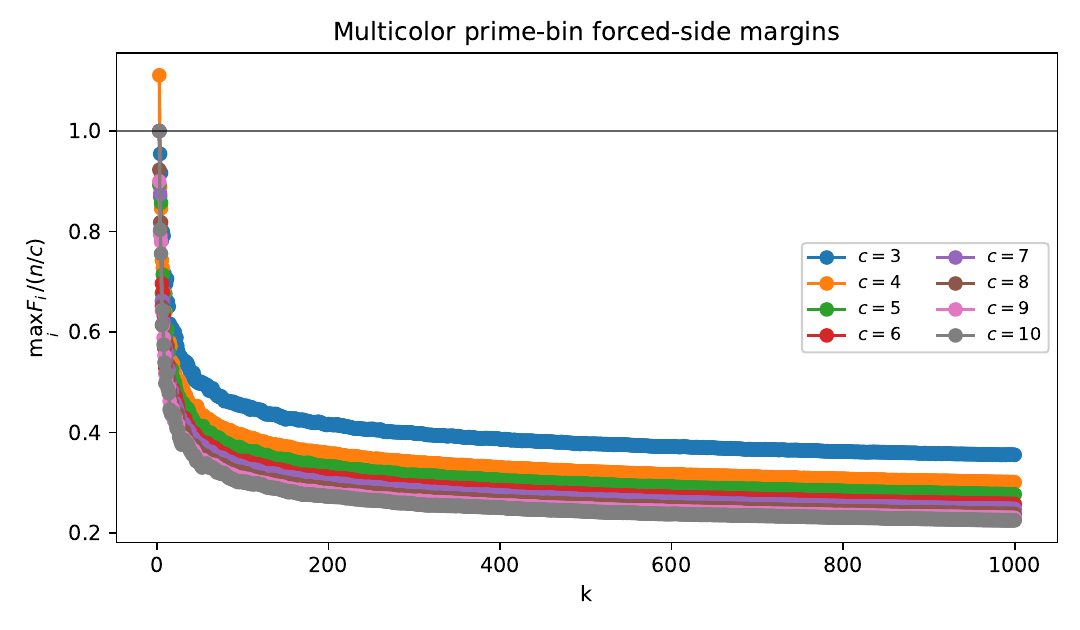}
\caption{For the start-$0$ round-robin certificate strategy in each color
count, the forced-side ratio drops below the balanced target after a small
initial range.  Feasibility is then certified by a max-flow assignment of
flexible vertices to color capacities.}
\label{fig:balanced-multicolor-margins}
\end{figure}

\subsection{Supplementary Boundary Diagnostics}
\label{app:boundary-diagnostics}

\begin{table}[htbp]
\centering
\caption{Small multicolor balanced endpoint defect map.  The ``exact''
columns record only cases where an exact MILP or SAT backend returned a
definitive answer; ``unknown'' means time or clique-enumeration limits were
reached and is not evidence of infeasibility, while ``--'' means no exact
case of that type is reported in the table.  The last column is the
independent certificate-family onset from the max-flow round-robin scan.}
\label{tab:balanced-multicolor-defect-map}
\small
\begin{tabular}{@{}rrrrr@{}}
\toprule
$c$ & exact no & exact yes & exact unknown & cert. from / tested \\
\midrule
3 & 3 & 4--5 & -- & 6 / 500 \\
4 & 3 & 4 & 5 & 8 / 500 \\
5 & 3 & -- & -- & 15 / 500 \\
6 & 3 & -- & -- & 16 / 500 \\
7 & 3 & -- & -- & 24 / 500 \\
8 & 3 & -- & -- & 28 / 500 \\
9 & 3 & -- & -- & 37 / 500 \\
10 & -- & -- & 3 & 53 / 500 \\
\bottomrule
\end{tabular}

\end{table}

\begin{table}[htbp]
\centering
\caption{Observed onset of the certificate-family regime for multicolor
balanced endpoints.  Each row summarizes max-flow feasibility inside the
round-robin prime-bin family at $n=p_{c(k-1)}-1$.  The rows $3\le c\le20$
were tested through $k=500$, and the rows $21\le c\le30$ through $k=400$.
An all-start comparison for $3\le c\le20$ and $k\le250$ selected the same
start-$0$ onset in every row.}
\label{tab:balanced-multicolor-phase}
\small
\begin{tabular}{@{}rrrrr@{}}
\toprule
colors $c$ & last failure & all-success from & tested through & $k_{\mathrm{on}}/(c\log c)$ \\
\midrule
3 & 5 & 6 & 500 & 1.82 \\
4 & 7 & 8 & 500 & 1.44 \\
5 & 14 & 15 & 500 & 1.86 \\
6 & 15 & 16 & 500 & 1.49 \\
7 & 23 & 24 & 500 & 1.76 \\
8 & 27 & 28 & 500 & 1.68 \\
9 & 36 & 37 & 500 & 1.87 \\
10 & 52 & 53 & 500 & 2.30 \\
11 & 55 & 56 & 500 & 2.12 \\
12 & 77 & 78 & 500 & 2.62 \\
13 & 91 & 92 & 500 & 2.76 \\
14 & 87 & 88 & 500 & 2.38 \\
15 & 95 & 96 & 500 & 2.36 \\
16 & 134 & 135 & 500 & 3.04 \\
17 & 153 & 154 & 500 & 3.20 \\
18 & 157 & 158 & 500 & 3.04 \\
19 & 182 & 183 & 500 & 3.27 \\
20 & 199 & 200 & 500 & 3.34 \\
21 & 201 & 202 & 400 & 3.16 \\
22 & 227 & 228 & 400 & 3.35 \\
23 & 243 & 244 & 400 & 3.38 \\
24 & 284 & 285 & 400 & 3.74 \\
25 & 334 & 335 & 400 & 4.16 \\
26 & 354 & 355 & 400 & 4.19 \\
27 & 354 & 355 & 400 & 3.99 \\
28 & 373 & 374 & 400 & 4.01 \\
29 & 365 & 366 & 400 & 3.75 \\
30 & 372 & 373 & 400 & 3.66 \\
\bottomrule
\end{tabular}

\end{table}

\begin{table}[htbp]
\centering
\caption{Expanded lower-bound certificate scan for shifted intervals.  A
certificate is an interval-adapted prime-bin coloring of
$\{m+1,\ldots,m+n\}$ using only witness primes present in the interval.}
\label{tab:shifted-lower-bound-summary}
\small
\begin{tabular}{@{}rrrr@{}}
\toprule
$k$ & shifts tested & shifts with certificate & max certified length \\
\midrule
3 & 499 & 0 & -- \\
4 & 499 & 0 & -- \\
5 & 499 & 0 & -- \\
6 & 499 & 0 & -- \\
7 & 499 & 0 & -- \\
\bottomrule
\end{tabular}

\end{table}

\Cref{tab:shifted-exact-summary} gives the complementary exact MILP frontier
for the small shifted instances where the direct formulation is still
tractable.

\begin{table}[htbp]
\centering
\caption{Exact shifted-interval MILP frontier.  Unknown means the instance was
not solved within the per-instance time limit and is not counted as a
threshold.}
\label{tab:shifted-exact-summary}
\small
\begin{tabular}{@{}rrrrr@{}}
\toprule
$k$ & shift range & exact & unknown & not found \\
\midrule
3 & 2--100 & 99 & 0 & 0 \\
4 & 2--100 & 99 & 0 & 0 \\
5 & 2--100 & 99 & 0 & 0 \\
6 & 2--100 & 42 & 57 & 0 \\
\bottomrule
\end{tabular}

\end{table}

\subsection{Exploratory Computation and Retained Diagnostics}
\label{app:exploratory-lessons}

The final proof is independent of exploratory computation.  The diagnostics
retained here explain how the direct search formulation differs from the
certificate proof and record which computational claims were superseded by the
exact theorem.

\subsubsection{Retained diagnostics}

\begin{itemize}
  \item The vertex-coloring and edge-coloring problems are genuinely
  different and must be separated at the definition level.
  \item The early exact values through $k=9$ were consistent with the final
  formula and remain useful sanity checks.
  \item The structural-diagnostics figure is valuable: it shows how a direct
  SAT encoding grows while the true proof certificate stays tiny
  (\Cref{fig:structural-diagnostics}).
  \item The comparison with recent Ramsey breakthroughs is useful when
  stated carefully: this paper does not improve classical Ramsey bounds, but
  it illustrates a sharp contrast between pseudorandom search and arithmetic
  certificates.
  \item The balanced-coloring objection led to a real side result: the exact
  two-color balanced endpoint.
\end{itemize}

\begin{figure}[htbp]
\centering
\includegraphics[width=0.76\linewidth]{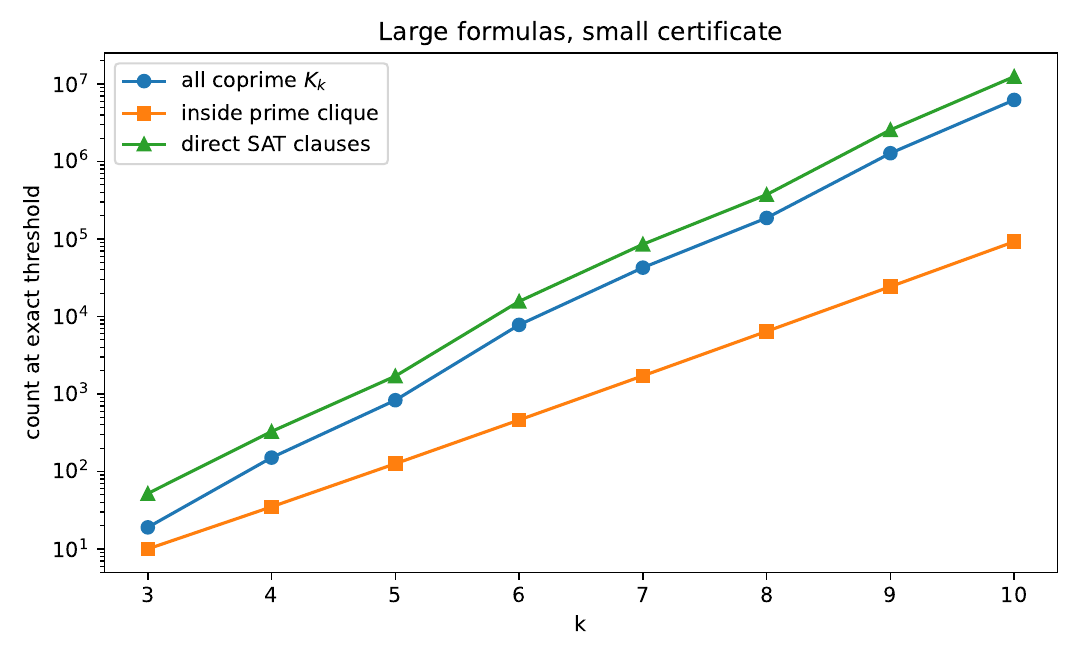}
\caption{Clique and clause counts grow rapidly, but the exact upper-bound
certificate is the single prime clique of size $2k-1$ at
$n=p_{2k-2}$.}
\label{fig:structural-diagnostics}
\end{figure}

\subsubsection{Superseded computational claims}

\begin{itemize}
  \item The claim $R_{\rm cop}(10)=53$ is false for the vertex-coloring
  two-color diagonal problem.
  \item Computing $R_{\rm cop}(11)$ is no longer a meaningful frontier; the
  theorem gives $R_{\rm cop}(11;2)=p_{20}=71$ immediately.
  \item General computational-hardness or fine-grained complexity claims are
  not part of this final theorem unless one changes the input model to
  arbitrary graphs or to edge-coloring targets outside the complete-witness
  transfer.
  \item Spectral and pseudorandomness diagnostics are explanatory, not proof
  ingredients.  The prime clique and divisor injection already determine the
  threshold.
  \item Primality patterns in finite tables should not be presented as
  evidence.  The values are prime because the theorem says they are indexed
  primes.
\end{itemize}

\subsection{Systematic Edge-Coloring Comparison}
\label{app:edge-comparison}

The edge-coloring coprime Ramsey problem is the closest public neighbor, so
it deserves a precise separation.  In the vertex version, a color class is a
set of integers.  To avoid a monochromatic $K_k$, it is enough to ensure
that each color class has pairwise-coprime packing number at most $k-1$.
Prime bins directly control that packing number.

In the edge version, the vertices of a candidate clique need not have a
single color.  Instead, every coprime pair inside the clique has its own
edge color.  A prime-bin assignment of vertices says nothing about whether
all edges among a set of pairwise coprime vertices are red or blue.  This is
why the reported edge-coloring values are already much larger:
\[
  \Redge(3;2)=11,\qquad
  \Redge(3;3)=53,\qquad
  \Redge(4;2)=59,
\]
while the vertex-coloring values are
\[
  \Rcop(3;2)=7,\qquad
  \Rcop(3;3)=13,\qquad
  \Rcop(4;2)=13.
\]

The distinction also explains why the two sets of reported values differ
without making the edge version mysterious.  The prime-bin vertex proof does
not color edges, but the prime-label pullback in \Cref{thm:edge-reduction}
does: it shows that all edge-coprime values are classical Ramsey values
viewed through the map $N\mapsto p_{N-1}$.  Computation can still contribute
to edge-coprime tables, but only to the extent that it contributes to the
underlying classical Ramsey table.

This also clarifies the status of the nearest public computations.  Towell's
repository and accompanying note \cite{towell-blog,towell-github} describe an
online computational project using SAT solvers and AI assistance, and report
the edge-coloring values $11$, $53$, and $59$ using SAT and extension checks.
Those observations use
the same coprime host graph and the same informal phrase ``coprime Ramsey'',
but they do not color vertices and do not contain the mixed partition formula
of Theorem~\ref{thm:main}.  Conversely, \Cref{thm:edge-reduction} gives a
short proof of exactly those clique edge-coloring values from the classical
Ramsey numbers $6$, $17$, and $18$.

\subsection{Why Random-Graph Intuition Fails}
\label{app:randomness}

One tempting but misleading analogy is to compare $G_n$ with an
Erd\H{o}s--R\'enyi graph of the same edge density.  The edge density of the
coprime graph tends to
\[
  \mathbb{P}(\gcd(a,b)=1)=\frac{1}{\zeta(2)}=\frac{6}{\pi^2}.
\]
A random graph $G(n,6/\pi^2)$ has clique number logarithmic in $n$ with high
probability.  The coprime graph has the explicit clique
\[
  \{1\}\cup\{p\le n:p\text{ prime}\},
\]
whose size is $\pi(n)+1\sim n/\log n$.  This is not a small deviation; it is
the structural feature that determines the Ramsey threshold.

This explains a central difference between the present problem and classical
Ramsey lower-bound constructions.  In random or pseudorandom constructions,
large cliques are scattered and controlled by concentration.  In $G_n$, the
dominant clique is named explicitly.  Once that clique is identified, the
upper bound is a pigeonhole argument and the lower bound is the matching
divisor-bin construction.

The exploratory computations included spectral and random-graph diagnostics.
Those diagnostics are still useful as intuition: the universal
vertex $1$, the prime clique, and repeated divisor patterns create strong
non-random signatures.  They should not, however, be used as proof of the
Ramsey value.  The proof is entirely combinatorial and depends only on
prime-divisor injection.

\subsection{Literature Search Notes}
\label{app:literature-search}

The novelty statement in the related-work section is deliberately modest.
It says that, within the checked public sources available through
2026-05-25, we did not find a fully overlapping vertex-coloring formula.  It
does not claim that a result can be proven absent from every private
manuscript or every unindexed web page.

The database comparison used these arXiv application programming interface
(API) queries:
\[
\begin{gathered}
\texttt{all:"coprime Ramsey"},\quad
\texttt{all:coprime AND all:Ramsey},\\
\texttt{all:"coprime graph" AND all:Ramsey},\quad
\texttt{all:"pairwise coprime" AND all:Ramsey},\\
\texttt{all:"vertex-coprime Ramsey"},\quad
\texttt{all:"R\_cop" AND all:Ramsey}.
\end{gathered}
\]
All returned zero arXiv entries.  Repository search found one near-name
project, Towell's computational exploration, which concerns edge colorings.
Semantic paper search returned some raw hits for broad phrases, but after
screening none contained both coprime and Ramsey terminology in a way
relevant to this problem.  \Cref{tab:prior-art} records this public-source
screening summary.

\begin{table}[htbp]
\centering
\caption{Public-source screening summary for exact or near-name overlap.}
\label{tab:prior-art}
\begin{tabular}{lrrr}
\toprule
Source & queries & raw hits & near-name hits \\
\midrule
arXiv API & 6 & 0 & 0 \\
GitHub repository API & 4 & 1 & 1 \\
\bottomrule
\end{tabular}

\end{table}

The closest prior items are therefore:
\begin{itemize}
  \item the $k=2$ chromatic-number observation for the coprime graph;
  \item Towell's edge-coloring coprime Ramsey project;
  \item extremal-set papers on one large subset with no many pairwise
  coprime integers.
\end{itemize}
We did not find any of these sources stating the mixed vertex-coloring formula
of Theorem~\ref{thm:main}.

\subsection{Why the Result Is Complete, and Where It Is Not}
\label{app:scope}

For the vertex-coloring problem on $G_n=[1,n]$, the value frontier is fully
closed:
\[
  \Rcop(k_1,\ldots,k_c)=p_{\sum_i(k_i-1)}
\]
for every number of colors and every demand vector.  Pushing
$\Rcop(k;2)$ to larger $k$ is therefore not a computational research goal;
it is a table-generation exercise.

The meaningful continuations are variants that break one of the two pillars
of the proof.

\paragraph{Edge colorings.}
The prime-bin construction colors vertices, not edges, but the edge variant
is resolved by \Cref{thm:edge-reduction}.  It is not a separate source of
new Ramsey numbers; it is a prime-index transform of the classical table.

\paragraph{Shifted intervals.}
Replacing $[n]$ by $\{m+1,\ldots,m+n\}$ removes vertex $1$ and the initial
prime clique.  The threshold should depend on local prime-factor structure,
not only on $\pi(n)$.

\paragraph{Balanced and density-constrained colorings.}
The balanced two-color diagonal endpoint is now closed by
\Cref{thm:balanced-exact}.  The next density-constrained questions are
multicolor balance, off-diagonal balance, and prescribed non-half densities.

\paragraph{Other arithmetic graphs.}
Graphs defined by common divisors, squarefree kernels, prescribed gcd sets,
or algebraic-integer coprimality may preserve enough divisor structure for
exact Ramsey thresholds while avoiding the complete collapse seen here.

\subsection{Logical Dependency Map}
\label{app:dependency-map}

The proof dependencies are:
\[
\text{prime-divisor injection}
\Longrightarrow
\omega(G_n)=\pi(n)+1
\Longrightarrow
\text{prime-clique upper bound};
\]
\[
\text{prime-bin capacities}
\Longrightarrow
\text{extremal coloring below }p_M
\Longrightarrow
\text{exact mixed formula}.
\]
All computations in the paper are downstream of this theorem.  They serve
three purposes: they reproduce finite values, explain why the original SAT
route appeared difficult, and test robustness variants such as balanced
coloring.  No computational result is used as a premise for
Theorem~\ref{thm:main}.

This dependency map also explains the order of presentation: the prime
clique and matching prime-bin coloring come first, while computation enters
only afterward as diagnostic evidence.

\subsection{Code and Data Availability}
\label{app:code}

The companion repository contains the scripts and generated data used for
the tables and figures in this report.  The source hierarchy separates the
active manuscript and reproducibility scripts from older exploratory files,
which are retained only for provenance and are not sources of final claims.

\end{document}